\documentclass[a4paper,12pt]{article}
\usepackage[utf8]{inputenc}
\usepackage[T1]{fontenc}
\usepackage{mathtools}
\usepackage[left=1.5cm,top=2.5cm,right=1.5cm,bottom=2.5cm]{geometry}

\usepackage{pgf,tikz,pgfplots}
\usepackage{mathrsfs}
\usepackage{setspace}
\usepackage{hyperref,nameref}
\usepackage{cite}
\usepackage{amsthm}
\usepackage{amsmath,amssymb}
\usepackage{enumerate}

\usepackage{appendix} 
\newtheorem{teo}{Theorem}[section] %El teo es el contador que me numera los lemas
\newtheorem*{teo*}{Theorem}
\newtheorem{lem}[teo]{Lemma} % ahora el teo me dao una numeración distinta a los lemas
\newtheorem{prop}[teo]{Proposition}
\newtheorem{cor}[teo]{Corollary}
\newtheorem{definicion}{Definition}[section]

\newcommand{\R}{\mathbb{R}}\newcommand{\Rn}{\R^n}\newcommand{\Rnn}{\R^{n\times n}}
\newcommand{\N}{\mathbb{N}}

\DeclareMathOperator{\diver}{div}

\DeclareMathOperator{\cof}{cof}

\DeclareMathOperator{\tr}{tr}
\DeclareMathOperator{\pv}{pv}

\DeclareMathOperator{\loc}{loc}

\newcommand{\weakc}{\rightharpoonup}
\renewcommand{\O}{\Omega}
\renewcommand{\a}{\alpha}

\usepackage{amsmath}
\usepackage{latexsym}
\usepackage{amssymb}
\def\XXint#1#2#3{{\setbox0=\hbox{$#1{#2#3}{\int}$}
\vcenter{\hbox{$#2#3$}}\kern-.5\wd0}}

\title{$\Gamma$-convergence of polyconvex functionals involving $s$-fractional gradients to their local counterparts}
\author{J. C. Bellido, J. Cueto, C. Mora-Corral}
\usepackage{graphicx}
\usepackage{wrapfig} %Para recolocar el gráfico poniendo letras alrededor

\begin{document}

\maketitle

\begin{abstract}
In this paper we study localization properties of the Riesz $s$-fractional gradient $D^s u$ of a vectorial function $u$ as $s \nearrow 1$.
The natural space to work with $s$-fractional gradients is the Bessel space $H^{s,p}$ for $0 < s < 1$ and $1 < p < \infty$.
This space converges, in a precise sense, to the Sobolev space $W^{1,p}$ when $s \nearrow 1$.
We prove that the $s$-fractional gradient $D^s u$ of a function $u$ in $W^{1,p}$ converges strongly to the classical gradient $Du$.
We also show a weak compactness result in $W^{1,p}$ for sequences of functions $u_s$ with bounded $L^p$ norm of $D^s u_s$ as $s \nearrow 1$.
Moreover, the weak convergence of $D^s u_s$ in $L^p$ implies the weak continuity of its minors, which allows us to prove a semicontinuity result of polyconvex functionals involving $s$-fractional gradients defined in $H^{s,p}$ to their local counterparts defined in $W^{1,p}$.
The full $\Gamma$-convergence of the functionals is achieved only for the case $p>n$.
\end{abstract}


% %\addcontensline{toc}{chapter}{Contents}
%\tableofcontents
% %\markboth{{\bf Contents}}{{\bf Contents}}

%\newpage

%\pagestyle{plain}

%\input{introduction}
%\input{Preliminaries}
%\input{Poincare_Sobolev_inequality}
%\input{Localization}
%\input{Compactness}
%\input{Weak_continuity_minors}
%\input{gamma_convergence_section}

\section[Introduction]{Introduction}

Fractional Sobolev spaces $W^{s,p}(\R^n)$, with $0<s<1$ and $1\le p<\infty$, are nowadays of central importance in the analysis of partial differential equations, both of local or nonlocal (or fractional) nature \cite{DPV}. These spaces provide, through the Gagliardo seminorm of a function $u$ defined as 
\[ [u]_{W^{s,p}(\R^n)}=\left(\int_{\R^n}\int_{\R^n} \frac{|u(x)-u(y)|^p}{|x-y|^{n+sp}}\,dx\,dy\right)^\frac{1}{p},\]
a measure of the fractional differentiability of order $s$ of the function $u$. Indeed, in \cite{BoBrMi2001} (see also \cite[Prop.\ 15.7]{Ponce2004}) it has been shown that for $p>1$,
\[\lim_{s \nearrow 1} (1-s)^\frac{1}{p}[u]_{W^{s,p}(\R^n)}=K(n,p)\| Du\|_{L^p(\R^n)}\]
for some constant $K(n,p)>0$.
This provides an interesting characterization of Sobolev spaces, which naturally leads to the question of the existence of a fractional differential object converging to the classical gradient as $s$ goes to $1$. The Riesz $s$-fractional gradient seems to be the right notion for such a differential object, and has been addressed by several authors in different situations in the recent years \cite{ShS2015,ShS2018,Schikorra2017,COMI2019,Silhavy2019,BeCuMC}.
The $s$-fractional gradient, $s\in(0,1)$, of an $L^1_{\loc}$ function $u:\R^n\rightarrow \R$  is defined as 
\begin{equation}\label{eq:Dsu}
D^su(x)=c_{n,s}\pv_x\int_{\mathbb{R}^n} \frac{u(x)-u(y)}{|x-y|^{n+s}}\frac{x-y}{|x-y|}dy ,
\end{equation}
whenever it makes sense, where $\pv_x$ stands for the principal value centred at $x$, and  $c_{n,s}$ is a suitable
normalizing constant. As mentioned above, the first reason for which this object deserves attention is the fact that $D^su$ converges to the classical gradient $Du$ as $s \nearrow 1$. Indeed, $D^su=D (I_{1-s}*u)$ for any $u\in C^{\infty}_c(\R^n)$ \cite{COMI2019}, where $I_{1-s}(\cdot)=\frac{c_{n,s}}{n+s-1} \left| \cdot \right|^{-n+1-s}$ is the classical Riesz potential \cite{Stein70}; therefore, applying Fourier transform, 
\[\widehat{D^su}(\xi) = 2\pi i\xi \widehat{I_{1-s}} (\xi) \hat{u} (\xi) = 2\pi i\xi \left| 2\pi\xi \right|^{s-1}\hat{u} (\xi),\]
which converges to $\widehat{Du} (\xi)$ as $s \nearrow 1$. 

Another remarkable reason why the fractional gradient seems to be the right differential object is given in \cite{Silhavy2019}: it is shown that for $s\in(0,1)$, definition \eqref{eq:Dsu} determines up to a multiplicative constant the unique object fulfilling some minimal consistency requirements from the physical and mathematical point of view, such as invariance under rotations and translations, $s$-homogeneity under dilations and some weak continuity properties.
Furthermore, an $s$-fractional divergence may be defined, satisfying a similar characterization \cite{Silhavy2019}, as well as the integration by parts formula; consequently, the $s$-fractional gradient and the $s$-fractional divergence are dual operators \cite{BeCuMC,COMI2019,DGLZ,MeS,Silhavy2019}.

The first reference dealing with this sort of differential objects seems to be \cite{horvath}. More recently, nonlocal gradients defined as integrals in bounded domains have been investigated in \cite{MeS} (see also \cite{DGLZ}), showing some vector calculus facts, such as the integration by parts formula, and localization results in various topologies when the {\it horizon} of interaction among particles vanishes.

When dealing with variational problems or fractional PDE involving the $s$-fractional gradient, there appears naturally the space $H^{s,p}(\Rn)$, for $0 < s < 1$ and $1\leq p< \infty$, defined as the completion of $C^{\infty}_c (\Rn)$ under the norm 
\[ \|u\|_{H^{s,p}(\R^n)}=\|u\|_{L^p(\R^n)}+\|D^su\|_{L^p(\R^n)}.\]
To be consistent with this definition of $H^{s,p}$ as a completion, given $u \in H^{s,p} (\Rn)$ we define $D^s u$ as the limit of $D^s u_j$ in $L^p$ when $\{ u_j \}_{j \in \N} \subset C^{\infty}_c (\Rn)$ is any sequence converging to $u$ in $H^{s,p}$.
We leave for a subsequent work the study of whether this definition for $D^s u$ coincides with Riesz' fractional gradient \eqref{eq:Dsu}. 

In \cite{ShS2015,ShS2018} the space $H^{s,p}$ is extensively studied, showing Sobolev-type and compact embedding theorems; they also provide existence results for scalar convex variational problems involving the $s$-fractional gradient, and for the derived fractional PDE obtained as the first order equilibrium condition. See also \cite{Schikorra2017} and the references therein for the case $p=1$. In \cite{COMI2019}, the space of functions whose $s$-fractional total variation is finite is considered and, subsequently, an $s$-fractional Caccioppoli perimeter is addressed. In \cite{BeCuMC}, the authors of this paper have addressed the study of polyconvex functionals depending on the $s$-fractional gradient, showing a fractional Piola identity, the weak continuity of the determinant of fractional gradients, a semicontinuity result for polyconvex functionals depending on $D^su$, and finally an existence theory in that situation.

In this paper we  continue with the study of polyconvex functionals depending on the $s$-fractional gradient by further exploring it through the study of its limit when $s \nearrow 1$. The main results of the paper are described as follows.
We prove the strong convergence in $L^p$ of $D^s u$ to $Du$ for functions $u\in W^{1,p}$, generalizing, and making the topology precise, the convergence mentioned above for smooth functions. Notice that this convergence is performed in the fractional parameter $s$ rather than in the {\it horizon}, as done in \cite{MeS}. This result is of interest in itself, as it provides a precise differential object converging to the distributional gradient. We also show a weak compactness result in $W^{1,p}$, establishing that if $\{ u_s \}$ is a sequence such that $\{ D^s u_s \}$ is bounded in $L^p$, then there exists a $u \in W^{1,p}$ such that $u_s$ converges strongly and $D^su_s$ converges weakly in $L^{p}$ 
as $s\nearrow 1$ to $u$ and $Du$, respectively. 
The weak convergence of the minors of $D^su_s$ to those of $Du$, whenever $D^su_s$ converges weakly in $L^p$ to $Du$, is also shown, and as a consequence, a new semicontinuity result for polyconvex functionals. Finally, we show that the family of vector variational problems based on minimization of
\[
\mathcal{I}_s (u)=\int_{\mathbb{R}^n}W(x,u(x),D^su(x)) \, dx, \qquad u \in H^{s,p} (\R^n, \R^n)
\] 
$\Gamma$-converges (see \cite{Braides}) to the functional
\[
\mathcal{I} (u)=\int_{\mathbb{R}^n}W(x,u(x),Du(x)) \, dx , \qquad u \in W^{1,p} (\R^n, \R^n)
\]
as $s \nearrow 1$, under the essential assumption of polyconvexity of $W(x,u, \cdot)$ \cite{dacorogna}; we also need the extra assumption $p>n$ for the $\Gamma$-convergence.
Other references dealing with $\Gamma$-convergence of variational functionals in the nonlocal setting are \cite{Pon} (in the context of $W^{s,p}$),  \cite{BeMCPe} (in nonlinear peridynamics) and \cite{MeD} (in linear and geometrically nonlinear peridynamics).

{The outline of the paper is as follows.
Section \ref{se:preliminaries} is a preliminary section where we collect several results on the fractional differential operators $D^s$ and $\diver^s$, as well as on the Bessel space $H^{s,p}$. We also provide a refinement on the Poincar\'e-Sobolev inequality in $H^{s,p}$ given in \cite{ShS2015}, with a constant independent of $s$  (Theorem \ref{th:Poincare}). In Section \ref{se:localization} we prove the localization of the fractional gradient; namely, that $D^s u$ converges to $Du$ in $L^p$ for a fixed $u \in W^{1,p}$ (Theorem \ref{Prop: convergencia del gradiente fraccionario al clásico}).
In Section \ref{se:compactness} we prove the compactness result: for any sequence $\{ u_s \}$ with a fixed complementary-value data and bounded $\{ D^s u_s \}$ in $L^p$, there exists $u\in W^{1,p}$ and a subsequence a strongly convergent to $u$ in $L^p$ and the $s$-fractional gradients weakly convergent to $Du$ in $L^p$ (Theorem \ref{Pr: weak convergence in s}). Section \ref{sc:weakcontinuity} is devoted to show the weak convergence of the minors of $D^s u_s$ to the minors of $Du$ when $D^su_s$ converges weakly to $Du$ in $L^p$ (Theorem \ref{th:wcontdet}). 
Finally, in Section \ref{se:Gamma} we prove a novel semicontinuity result for polyconvex functionals (Theorem \ref{prop:lsc}) and the $\Gamma$-convergence result of $\mathcal{I}_s$ to $\mathcal{I}$ under the assumption of polyconvexity (Theorem \ref{th:Gconvergence}) and convexity (Theorem \ref{th:GconvergenceC}).

\section[Preliminaries]{Preliminaries on Bessel spaces}\label{se:preliminaries} 

In this section we introduce the fractional differential operators $D^s$ and $\diver^s$, and the Bessel space $H^{s,p}$, together with some embedding theorems and a Poincar\'e-Sobolev inequality in this fractional context.

\subsection{Fractional differential operators}

We state the definition of the $s$-fractional gradient and divergence.
%perform a limit passage. However, in the case of the $s$-fractional divergence operator it is still needed even for smooth functions,
%so we proceed with the definition of principal value.
Given a function $f : \Rn \to \R$ and $x \in \Rn$ such that $f \in L^1 (B(x,r)^c)$ for every $r>0$, we define the principal value centered at $x$ of $\int_{\Rn} f$, denoted by
\[
\pv_{x} \int_{\mathbb{R}^n}f \quad \text{or} \quad \pv_{x}\int f ,
\]
as
\begin{equation*}
\lim_{r \rightarrow 0} \int_{B(x,r)^c} f,
\end{equation*}
whenever this limit exists.
We have denoted by $B(x,r)$ the open ball centered at $x$ of radius $r$, and by $B(x,r)^c$ its complement.
As most integrals in this work are over $\Rn$, we will use the symbol $\int$ as a substitute for $\int_{\Rn}$.

In order to avoid the principal value in \eqref{eq:Dsu}, we first establish the following definition for $C^{\infty}_c$ functions and then we extend it by density. 
The following definitions of $s$-fractional gradient and divergence are adapted from \cite{ShS2015,ShS2018,MeS,BeCuMC}.
Recall that $\Gamma$ denotes Euler's gamma function. 

\begin{definicion}\label{de:Ds}
	Let $0<s<1$ and set
	\begin{equation*}
	c_{n,s} := (n+s-1) \frac{\Gamma \left(\frac{n+s-1}{2} \right)}{\pi^{\frac{n}{2}} \, 2^{1-s} \, \Gamma \left(\frac{1-s}{2}\right)} .
	\end{equation*}
	
	\begin{enumerate}[a)]
		\item\label{item:Dsu} Let $u \in C^{\infty}_c (\Rn)$.
		We define $D^s u : \Rn \to \Rn$ as
		\begin{equation}\label{eq:DsuCc}
		D^s u (x):= c_{n,s} \int \frac{u(x)-u(y)}{|x-y|^{n+s}}\frac{x-y}{|x-y|}dy .
		\end{equation}
		
		\item Let $\phi \in C^{\infty}_c (\Rn, \Rn)$.
		We define $\diver^s \phi : \Rn \to \R$ as
		\begin{equation*}%\label{eq:divsuCc}
		\diver^s \phi (x):= -c_{n,s} \pv_x \int\frac{\phi(x)+\phi(y)}{|x-y|^{n+s}}\cdot\frac{x-y}{|x-y|}dy .
		\end{equation*}
		
	\end{enumerate}
\end{definicion}

The integral \eqref{eq:DsuCc} is easily seen to be absolutely convergent for all $x \in \Rn$.
%On the other hand, as a consequence of \cite[Lemma 3.5]{BeCuMC} (an analogue of \cite[Lemma 2.3]{MeS}), we have that $\diver^s \phi(x)$ is well defined if and only if so is $D^s \phi_j(x)$ ($j=1, \ldots, n$), $\phi_j$ being the $j$-th component of $\phi$.
Moreover, by odd symmetry,
\begin{equation}\label{eq:divabsconv}
 -c_{n,s} \pv_x \int\frac{\phi(x)+\phi(y)}{|x-y|^{n+s}}\cdot\frac{x-y}{|x-y|}dy = c_{n,s} \int\frac{\phi(x) - \phi(y)}{|x-y|^{n+s}}\cdot\frac{x-y}{|x-y|} dy, 
\end{equation}
and this last integral is absolutely convergent.
Furthermore, $D^s u \in L^q (\Rn, \Rn)$ and $\diver^s \phi \in L^q (\Rn, \Rn)$ for all $q \in [1, \infty]$ for smooth and compactly supported $u$ and $\phi$; see \cite[Lemma 3.1]{BeCuMC}, if necessary.

Definition \ref{de:Ds}\,\emph{\ref{item:Dsu})} naturally extends to $u \in C^{\infty}_c (\Rn, \R^n)$ by replacing \eqref{eq:DsuCc} with
\begin{equation}\label{eq:Dsuvector}
D^s u (x):= c_{n,s} \int \frac{u(x)-u(y)}{|x-y|^{n+s}} \otimes \frac{x-y}{|x-y|}dy .
\end{equation}
Here, $\otimes$ stands for the usual tensor product of vectors.

A crucial fact is the following fractional fundamental theorem of Calculus  \cite[Th.\ 3.11]{COMI2019} (see also \cite[Th.\ 1.12]{ShS2015} or \cite[Prop.\ 15.8]{ponce_book}).
\begin{teo} \label{Th: Fractional Fundamental Theorem of Calculus}
	Let $0<s<1$. For every $u \in C_c^{\infty}(\Rn)$ and every $x \in \Rn$ we have that
	\begin{equation*}
	u(x) = c_{n,-s} \int D^s u(y) \cdot \frac{x-y}{|x-y|^{n-s+1}}dy.
	\end{equation*}
\end{teo}

The operators $D^s$ and $\diver^s$ enjoy the following duality property, which is a nonlocal integration by parts whose proof can be found in \cite[Theorem 3.6]{BeCuMC} (see also \cite{COMI2019,MeS,Silhavy2019}).

\begin{lem} \label{Integracion por partes NL}
	Let $0<s<1$.
	Then, for all $u \in C^{\infty}_c (\Rn)$ and $\phi \in C^{\infty}_c (\Rn, \Rn)$ we have
	\begin{equation*}%\label{eq:intparts}
	\int D^s u(x) \cdot \phi(x) \, dx = -\int u(x) \diver^s \phi(x) \, dx.
	\end{equation*}
\end{lem}
%\begin{proof}
%	As $u \in C^{\infty}_c (\Rn)$ and $\phi \in C^{\infty}_c (\Rn, \Rn)$, they satisfy the hypotheses of \cite[Theorem 3.6]{BeCuMC} and thus \eqref{eq:intparts} holds.
%\end{proof}

We now extend Definition \ref{de:Ds} to a broader class of functions.

\begin{definicion}\label{de:Ds2}
	Let $0<s<1$ and $1 \leq p < \infty$.
	
	\begin{enumerate}[a)]
		\item
		Let $u \in L^p (\Rn)$ be such that there exists a sequence of $\{ u_j \}_{j \in \N} \subset C^{\infty}_c (\Rn)$ converging to $u$ in $L^p (\Rn)$ and for which $\{ D^s u_j \}_{j \in \N}$ is a Cauchy sequence in $L^p (\Rn, \Rn)$.
		We define $D^s u$ as the limit in $L^p (\Rn, \Rn)$ of $D^s u_j$ as $j \to \infty$.
		
		\item\label{item:divsu}
		Let $\phi \in L^p (\Rn, \Rn)$ be such that there exists a sequence of $\{ \phi_j \}_{j \in \N} \subset C^{\infty}_c (\Rn, \Rn)$ converging to $\phi$ in $L^p (\Rn, \Rn)$ and for which $\{ \diver^s \phi_j \}_{j \in \N}$ is a Cauchy sequence in $L^p (\Rn)$.
		We define $\diver^s \phi$ as the limit in $L^p (\Rn)$ of $\diver^s \phi_j$ as $j \to \infty$.
	\end{enumerate}
\end{definicion}

%The map $D^s u$ is called the $s$-fractional gradient of $u$, while $\diver^s \phi$ is called the $s$-fractional divergence of $\phi$.
Of course, for a $u \in L^p (\Rn, \R^n)$, the definition of $D^s u$ is analogous taking into account \eqref{eq:Dsuvector}.

The following result shows that the above definitions of $D^s u$ and $\diver^s \phi$ are independent of the sequences $\{ u_j \}_{j \in \N}$ and $\{ \phi_j \}_{j \in \N}$, respectively, and of the exponent $p$.

\begin{lem}\label{le:Duswell}
	Let $0<s<1$ and $1 \leq p, q < \infty$.
	\begin{enumerate}[a)]
		\item\label{item:Dsuindependent}
		Let $u \in L^p (\Rn) \cap L^q (\Rn)$ be such that there exist sequences $\{ u_j \}_{j \in \N} $ and $\{ v_j \}_{j \in \N}$ in $C^{\infty}_c (\Rn)$ such that $u_j \to u$ in $L^p (\Rn)$ and $v_j \to u$ in $L^q (\Rn)$, and for which $\{ D^s u_j \}_{j \in \N}$ converges to some $U$ in $L^p (\Rn, \Rn)$ and $\{ D^s v_j \}_{j \in \N}$ converges to some $V$ in $L^q (\Rn, \Rn)$.
		Then $U=V$.
		
		\item\label{item:divsphiindependent}
		Let $\phi \in L^p (\Rn, \Rn) \cap L^q (\Rn, \Rn)$ be such that there exist sequences $\{ \phi_j \}_{j \in \N} $ and $\{ \theta_j \}_{j \in \N}$ in $C^{\infty}_c (\Rn, \Rn)$ such that $\phi_j \to u$ in $L^p (\Rn, \Rn)$ and $\theta_j \to u$ in $L^q (\Rn, \Rn)$, and for which $\{ \diver^s \phi_j \}_{j \in \N}$ converges to some $\Phi$ in $L^p (\Rn)$ and $\{ \diver^s \theta_j \}_{j \in \N}$ converges to some $\Theta$ in $L^q (\Rn, \Rn)$.
		Then $\Phi = \Theta$.
	\end{enumerate}
\end{lem}
\begin{proof}
	We prove \emph{\ref{item:Dsuindependent})}, the proof of \emph{\ref{item:divsphiindependent})} being analogous.
	
	Let $\phi \in C^{\infty}_c (\Rn, \Rn)$.
	Then, by Lemma \ref{Integracion por partes NL},
	\[
	\int U \cdot \phi = \lim_{j \to \infty} \int D^s u_j \cdot \phi = - \lim_{j \to \infty} \int u_j \diver^s \phi = - \int u \diver^s \phi
	\]
	and, analogously,
	\[
	\int V \cdot \phi = - \int u \diver^s \phi .
	\]
	Thus, 
	\[
	\int U \cdot \phi = \int V \cdot \phi
	\]
	for all $\phi \in C^{\infty}_c (\Rn, \Rn)$, whence $U=V$.
\end{proof}

We end this subsection by showing two uniform bounds on the constant $c_{n,s}$ with respect to $s$.
We denote by $\omega_n$ the volume of the unit ball in $\R^n$.

\begin{lem}\label{Lemma: bound of constant c_{n,s}}
	Let $n \in \mathbb{N}$. Consider the function $c_{n,\cdot}:[-1,1] \rightarrow [0, \infty)$, defined as
	\[
	c_{n,s} = \begin{cases}
	\frac{\Gamma\left(\frac{n+s+1}{2}\right)}{\pi^{\frac{n}{2}}2^{-s}\Gamma\left( \frac{1-s}{2} \right)} & \text{if } -1 \leq s < 1 , \\
	0 & \text{if } s = 1 .
	\end{cases}
	\]
	Then
	\[
	\sup_{s \in [-1, 1]} c_{n,s}< \infty , \quad \sup_{s \in [-1, 1)}  \frac{c_{n,s}}{1-s}< \infty \quad \text{and} \quad \lim_{s\nearrow 1} \frac{c_{n,s}}{1-s} = \frac{1}{\omega_n}.
	\]
	%
	%Additionally, if $n=1$, we have that for every $s_0>0$ fixed, \eqref{eq: acotación c_{n,s}} holds for every $s$ such that $ |s| \in [s_0,1]$.
	% Furthermore, if $s \in[-1,0]$, there exists another constant $\tilde{C}$ such that
	% \[
	%0<\tilde{C}< |c_{n,s}|.
	%\]
\end{lem}
\begin{proof}
	The function $c_{n, \cdot}$ is clearly continuous in $[-1,1)$.
	As $\Gamma(z)\rightarrow +\infty$ as $z\rightarrow 0^+$, we obtain that $c_{n, \cdot}$ is also continuous in $[-1,1]$.
	Now, using the property $z \Gamma(z) = \Gamma(z+1)$ for $z >0$, we find that
	\[
	\frac{c_{n,s}}{1-s} = \frac{\Gamma(\frac{n+s+1}{2})}{\pi^{\frac{n}{2}}2^{1-s}\Gamma(\frac{3-s}{2})} ,
	\]
	and, hence, the function $s \mapsto \frac{c_{n,s}}{1-s}$ is continuous in $[-1,1)$ and admits a continuous extension to $[-1,1]$.
	In fact, 
	\[
	\lim_{s\nearrow 1} \frac{c_{n,s}}{1-s} = \frac{\Gamma(1 + \frac{n}{2})}{\pi^{\frac{n}{2}}} = \frac{1}{\omega_n} .
	\]
	The conclusion follows.
	%Finally, it is straightforward to see that when $s\in [-1,0]$, neither the numerator becomes $0$ or de denominator becomes $\infty$ and then there exists $\tilde{C}$ such that for every $s \in [-1,0]$
	%  \[
	% 0<\tilde{C}< |c_{n,s}|.
	% \]
\end{proof}

\subsection{Bessel spaces}

Let $0<s<1$ and $1 \leq p < \infty$.
For the sake of simplicity, we denote the norm in both $L^p(\Rn)$ and $L^p(\Rn,\Rn)$ by $\left\| \cdot \right\|_p$. Given $u \in C^{\infty}_c (\Rn)$, we define $\left\| \cdot \right\|_{H^{s,p}(\Rn)}$ as
\[
\left\| u \right\|_{H^{s,p}(\Rn)} = \left\| u \right\|_p + \left\| D^s u \right\|_p,
\]
which is easily seen to be a norm.
We define the space $H^{s,p}(\Rn)$ as the completion of $C^{\infty}_c (\Rn)$ under the norm $\left\| \cdot \right\|_{H^{s,p}(\Rn)}$, and extend accordingly the definition of $\left\| \cdot \right\|_{H^{s,p}(\Rn)}$ to $H^{s,p}(\Rn)$.
An analogous definition can be done for the space $H^{s,p} (\Rn, \R^n)$.

This is the definition given in \cite{ShS2015} (see also \cite{COMI2019}).
We leave for a future work the issue of whether this definition coincides with the set of $u \in L^p(\Rn)$ such that $D^s u$, as defined in \eqref{eq:Dsu}, exists a.e.\ and is in $L^p(\Rn, \Rn)$, which is the definition given in \cite{ShS2018}, without showing that both definitions are actually equivalent.

What was established in \cite[Th.\ 1.7]{ShS2015} is the remarkable fact of the identification of $H^{s,p}$ with the classical Bessel potential spaces (see \cite{Adams, Stein70, RunSi}).

For a $\phi \in H^{s,p} (\Rn, \Rn)$ there is a natural relation between $D^s \phi$ and $\diver^s \phi$.

\begin{lem}\label{le:div}
	Let $0 < s < 1$ and $1 \leq p < \infty$.
	Let $\phi \in H^{s,p} (\Rn, \Rn)$.
	Then $\diver^s \phi$ is well defined and $\tr D^s \phi = \diver^s \phi$ a.e.
\end{lem}
\begin{proof}
	Let $\{ \phi_j \}_{j \in \N} \subset C^{\infty}_c (\Rn, \Rn)$ be a sequence converging to $\phi$ in $L^p (\Rn, \Rn)$ such that $\{ D^s \phi_j \}_{j \in \N}$ converges to $D^s \phi$ in $L^p (\Rn, \Rnn)$.
	By linearity, $\tr D^s \phi_j \to \tr D^s \phi$ in $L^p (\Rn)$ as $j \to \infty$.
	In view of Definition \ref{de:Ds2}\,\emph{\ref{item:divsu})} and Lemma \ref{le:Duswell}\,\emph{\ref{item:divsphiindependent})}, it suffices to show that $\tr D^s \phi_j = \diver^s \phi_j$ for all $j \in \N$.
	Having in mind that the integrals of \eqref{eq:Dsuvector} and of the right hand side of \eqref{eq:divabsconv} are absolutely convergent, we obtain that
	\begin{align*}
	\tr D^s \phi_j (x) &= c_{n,s} \tr\left(\int  \frac{\phi_j(x) - \phi_j(y)}{|x-y|^{n+s}} \otimes \frac{x-y}{|x-y|}dy\right) = c_{n,s} \int \tr \left(\frac{\phi_j(x) - \phi_j(y)}{|x-y|^{n+s}} \otimes \frac{x-y}{|x-y|} \right) dy \\ 
	&= c_{n,s} \int \frac{\phi_j(x) - \phi_j(y)}{|x-y|^{n+s}} \cdot \frac{x-y}{|x-y|} dy = \diver^s \phi_j (x) ,
	\end{align*}
	which concludes the proof.
% Notice that, since at this point we are dealing with smooth functions and therefore absolutely convergent integrals, it is possible to interchange the trace operator with the integral of the principal value.
\end{proof}

The integration by parts formula of Lemma \ref{Integracion por partes NL} can be extended to $H^{s,p}$ as follows.

\begin{lem}\label{le:divgrad}
	Let $0<s<1$ and $1 < p < \infty$.
	Then, for all $u \in H^{s,p} (\Rn)$ and $\phi \in H^{s,p'} (\Rn, \Rn)$ we have
	\begin{equation*}%\label{eq:intparts}
	\int D^s u(x) \cdot \phi(x) \, dx = -\int u(x) \diver^s \phi(x) \, dx.
	\end{equation*}
\end{lem}
\begin{proof}
	Let $\{ u_j \}_{j \in \N} \subset C^{\infty}_c (\Rn)$ be a sequence converging to $u$ in $H^{s,p} (\Rn)$, and let $\{ \phi_j \}_{j \in \N} \subset C^{\infty}_c (\Rn, \Rn)$ be a sequence converging to $\phi$ in $H^{s,p'} (\Rn, \Rn)$.
	Then the following convergences hold as $j \to \infty$:
	\begin{align*}
	& u_j \to u \text{ in } L^p (\Rn), \quad D^s u_j \to D^s u \text{ in } L^p (\Rn, \Rn), \quad \phi_j \to \phi \text{ in } L^{p'} (\Rn, \Rn), \\
	& D^s \phi_j \to D^s \phi \text{ in } L^{p'} (\Rn, \Rnn) , \quad \tr D^s \phi_j \to \tr D^s \phi \text{ in } L^{p'} (\Rn) , \quad \diver^s \phi_j \to \diver^s \phi \text{ in } L^{p'} (\Rn) ,
	\end{align*}
	the last convegence due to Lemma \ref{le:div}.
	As a consequence,
	\begin{equation}\label{eq:DsdivL1}
	D^s u_j \cdot \phi_j \to D^s u \cdot \phi \quad \text{and} \quad u_j \diver^s \phi_j \to u \diver^s \phi \qquad \text{in } L^1 (\Rn) .
	\end{equation}
	By Lemma \ref{Integracion por partes NL}, for each $j \in \N$,
	\[
	\int D^s u_j (x) \cdot \phi_j (x) \, dx = -\int u_j (x) \diver^s \phi_j(x) \, dx .
	\]
	This equality and the convergences \eqref{eq:DsdivL1} readily imply the conclusion.
\end{proof}

%We start by recalling the definition of principal value.
%Given a function $f : \Rn \to \R$ and $x \in \Rn$ such that $f \in L^1 (B(x,r)^c)$ for every $r>0$, we define the principal value centered at $x$ of $\int_{\Rn} f$, denoted by
% \[
% \pv_{x} \int_{\mathbb{R}^n}f \quad \text{or} \quad \pv_{x}\int f ,
% \]
% as
% \begin{equation*}
% \lim_{r \searrow 0} \int_{B(x,r)^c} f,
% \end{equation*}
% whenever this limit exists.
% We have denoted by $B(x,r)$ the open ball centered at $x$ of radius $r$, and by $B(x,r)^c$ its complement in $\Rn$.

%Now we recall the version of the Hardy-Littlewood-Sobolev inequality from \cite[Ch.\ V, Th.\ 1]{Stein70}.
%\begin{prop} \label{Prop: Hardy-Littlewood-Sobolev inequality}
%Let $0<\alpha<n$ and $1< p<\infty$. Let $q$ be such that $\frac{1}{q}=\frac{1}{p}-\frac{\alpha}{n}$.
%Then there exists $C>0$ depending only on $p$ and $q$ such that for every $u \in L^p (\R^n)$,
% \begin{equation*}
%\lVert I_{\alpha}u \rVert_{L^q(\mathbb{R}^n)} \leq C \lVert u\rVert_p.
%\end{equation*}
%\end{prop}

\subsection{Embeddings of Bessel spaces}

It is known that, for $0<s<1$, the space $W^{1,p}(\mathbb{R}^n)$ embeds continuosly into $H^{s,p}(\mathbb{R}^n)$ \cite[Ch.\ 7]{Adams}. In the next result we prove again this result but writing the dependence of the embedding constant with respect to $s$.

\begin{prop} \label{Lemma: inequality with W1p}
Let $1 \leq p < \infty$. Then, there exists a constant $C=C(n,p)>0$ such that for all $u \in W^{1,p}(\mathbb{R}^n)$ and $0<s<1$,
	\begin{equation*}
	\lVert D^su  \rVert_p \leq \frac{C}{s} \lVert u\rVert_{W^{1,p}(\mathbb{R}^n)},
	\end{equation*}
\end{prop}
\begin{proof}
	By density, it is enough to prove the inequality for $u \in C^{\infty}_c (\Rn)$.
	For all $x \in \Rn$,
	\begin{equation}\label{eq:paso1}
	\left| D^s u (x) \right| \leq c_{n,s}  \left( A(x) + B(x) \right)
	\end{equation}
	with
	\[
	A (x) := \int_{B(x,1)} \frac{|u(x)-u(y)|}{|x-y|^{n+s}} \, dy , \qquad B (x) := \int_{B(x,1)^c} \frac{|u(x)-u(y)|}{|x-y|^{n+s}} \, dy ,
	\]
	so
	\begin{equation*}%\label{eq:DsuAB}
	\left\| D^s u \right\|_p \leq c_{n,s} \left( \left\| A \right\|_p + \left\| B \right\|_p \right) .
	\end{equation*}
	Note that
	\begin{equation*}%\label{eq:paso2}
	A (x) = \int_{B(0,1)} \frac{\left| u(x+h)-u(x) \right|}{|h|^{n+s}} dh , \qquad B (x) =  \int_{B(0,1)^c} \frac{\left| u(x+h)-u(x) \right|}{|h|^{n+s}} dh .
	\end{equation*}
	Applying Minkowski's integral inequality (see, e.g., \cite[App.\ A.1]{Stein70}) we obtain
	\begin{equation*}%\label{eq:paso3}
	\left\| A \right\|_p \leq \int_{B(0,1)} \left( \int \frac{|u(x+h)-u(x)|^p}{|h|^{(n+s)p}} \, dx \right)^{\frac{1}{p}} \, dh .
	\end{equation*}
	Now, for all $h \in B(0,1) \setminus \{ 0 \}$,
	\[
	\left( \int \frac{|u(x+h)-u(x)|^p}{|h|^{(n+s)p}} \, dx \right)^{\frac{1}{p}} = \frac{1}{|h|^{n+s}} \left(\int \left| u(x+h)-u(x) \right|^p  dx\right)^{\frac{1}{p}} \leq \frac{1}{|h|^{n+s-1}} \left\| D u \right\|_p ,
	\]
	thanks to a classic inequality (see, e.g., \cite[Prop.\ 9.3]{Brezis} and notice that it is still valid for $p=1$).
	Therefore,
	\begin{equation}\label{eq:Ap}
	\left\| A \right\|_p \leq \left\| D u \right\|_p \int_{B(0,1)} \frac{1}{|h|^{n+s-1}} = \frac{\sigma_{n-1}}{1-s} \left\| D u \right\|_p ,
	\end{equation}
	where $\sigma_{n-1}$ is the area of the unit sphere of $\Rn$.
	
	As for $B$, we first notice that for all $x \in \Rn$, by H\"older's inequality
	\begin{equation*}
	\begin{split}
	B(x) & \leq \int_{B(0,1)^c} \frac{\left| u(x+h) \right|}{|h|^{n+s}} dh + \int_{B(0,1)^c} \frac{\left| u(x) \right|}{|h|^{n+s}} dh \\
	& \leq \left( \int_{B(0,1)^c} \frac{\left| u(x+h) \right|^p}{|h|^{n+s}} dh \right)^{\frac{1}{p}} \left( \int_{B(0,1)^c} \frac{1}{|h|^{n+s}} dh \right)^{\frac{1}{p'}} + \left| u(x) \right| \int_{B(0,1)^c} \frac{1}{|h|^{n+s}} dh \\
	& = \left( \int_{B(0,1)^c} \frac{\left| u(x+h) \right|^p}{|h|^{n+s}} dh \right)^{\frac{1}{p}} \left( \frac{\sigma_{n-1}}{s} \right)^{\frac{1}{p'}} + \left| u(x) \right| \frac{\sigma_{n-1}}{s} ,
	\end{split}
	\end{equation*}
	so, by Fubini's theorem,
	\begin{equation}\label{eq:Bp}
	\left\| B \right\|_p \leq \left( \frac{\sigma_{n-1}}{s} \right)^{\frac{1}{p'}} \left( \int \int_{B(0,1)^c} \frac{\left| u(x+h) \right|^p}{|h|^{n+s}} dh \, dx \right)^{\frac{1}{p}} + \frac{\sigma_{n-1}}{s} \left\| u \right\|_p = 2 \frac{\sigma_{n-1}}{s} \left\| u \right\|_p .
	\end{equation}
	Putting together \eqref{eq:paso1}, \eqref{eq:Ap} and \eqref{eq:Bp}, we obtain
	\begin{equation*}
	\left\| D^s u \right\|_p \leq c_{n,s} \sigma_{n-1}  \left( \frac{1}{1-s} \left\| Du \right\|_p + \frac{2}{s} \left\| u \right\|_p \right) 
	\end{equation*}
	and, thanks to Lemma \ref{Lemma: bound of constant c_{n,s}}, the proof is finished.
\end{proof}

%We have denoted by $W^{s,p}$ the classical fractional Sobolev space.
%They will not be used in this paper, but they were mentioned in Proposition \ref{Theorem properties H^{s,p}} to help locate the spaces $H^{s,p}$ in a scale of regularity.
%We have also denoted by $C^{0,\mu}$ the space of H\"older continuous functions of exponent $\mu$.
%Among these properties we highlight the density of smooth functions as well as the aforementioned nesting property with the Gagliardo fractional Sobolev spaces. This last property in particular also provides us with the embedding of the $H^{s,p}(\Rn)$ in a decreasing order of $s$. Nevertheless, as it is not specified the independence of $s$ in the continuous embedding, will provide Proposition \ref{Lemma: immersion of the Hsp} where it will be detailed such independence as well as in other results throughout this paper.\\

Given an open set $\O\subset \Rn$ we define the subspace $H^{s,p}_0 (\O)$ as the closure of $C^{\infty}_c (\O)$ in $H^{s,p} (\Rn)$; of course, with this definition we are extending each function of $C^{\infty}_c (\O)$ by zero in $\O^c$.
In addition, given $g \in H^{s,p} (\Rn)$ we define the affine subspace $H^{s,p}_g (\Rn)$ as $g + H^{s,p}_0 (\O)$.
The affine subspace $W_g^{1,p}(\Omega)$ is defined in a similar way for $g \in W^{1,p} (\Rn)$.
We note that in \cite{ShS2018} the space $H_g^{s,p}(\O)$ was defined as the set of $u\in H^{s,p} (\Rn)$ such that $u=g$ in $\O^c$.
We leave for a future work the issue of the equality of both definitions, but we note that, for a $u \in H^{s,p}_g (\O)$ according to our definition, we trivially have that $u=g$ in $\O^c$.
In particular, the following compact embedding of $H^{s,p}_g (\O)$ into $L^q (\Rn)$ (see \cite[Th.\ 2.2]{ShS2018}; the formulation is adapted from \cite[Th.\ 2.3]{BeCuMC}) remains true.
In what follows we set $p_s^*=\frac{pn}{n-sp}$.

\begin{teo} \label{Bessel embedding theorem}
	Set $0<s<1$ and $1<p<\infty$. Let $\Omega\subset \mathbb{R}^n$ be open and bounded and $g \in H^{s,p}(\mathbb{R}^n)$. Then for any sequence $\{u_j \}_{j \in \N} \subset  H_{g}^{s,p}(\Omega)$
	such that
	\begin{equation*}
	u_j \rightharpoonup u \quad \text{in } H^{s,p} (\Rn),
	\end{equation*}
	for some $u \in H^{s,p} (\Rn)$, one has $u \in H^{s,p}_g (\O)$ and
\begin{enumerate}[a)]
\item $u_j - g \rightarrow u - g$ in $L^q (\Rn)$ for every $q$ satisfying
	\[
	\begin{cases}
	q\in \left[1, p_s^* \right) & \text{if } sp<n , \\
	q\in[1, \infty) & \text{if } sp = n , \\
	q\in[1, \infty] & \text{if } sp > n ,
	\end{cases}
	\]

\item $u_j \to u$ in $L^q (\Rn)$ for every $q$ satisfying
	\[
	\begin{cases}
	q\in \left[p, p_s^* \right) & \text{if } sp<n , \\
	q\in[p, \infty) & \text{if } sp = n , \\
	q\in[p, \infty] & \text{if } sp > n .
	\end{cases}
	\]
\end{enumerate}
\end{teo}

\subsection{Poincar\'e-Sobolev inequality for $H_0^{s,p}(\O)$}

In this section we prove the Poincar\'e-Sobolev inequality in $H^{s,p}$.
This result is known (see \cite[Th.\ 1.8]{ShS2015}), but for the analysis of this work it is crucial to trace the dependence of the Poincar\'e-Sobolev constant on $s$, in the case of bounded domains.
The proof we provide uses some ideas of \cite[Lemma 7.12]{GiTr01}.

\begin{teo}\label{th:Poincare}
	Let $\O \subset \Rn$ be open and bounded.
	Then there exists $C = C (n,\O)$ such that for all $0 < s < 1$, $1 < p < \infty$ and
	$u \in H^{s,p}_0 (\O)$,
	\[
	\left\| u \right\|_{L^p (\O)} \leq \frac{C}{s} \left\| D^s u \right\|_p .
	\]
\end{teo}
\begin{proof}
	By density, it is enough to prove the inequality for $u \in C^{\infty}_c (\O)$.
	Let $R \in \R$, to be specified later, such that  
	\begin{equation}\label{eq:R}
	R \geq 1 , \qquad \O \subset B (0, R) .
	\end{equation}
	Define $\O_1 := B (0, 2 R)$.
	
	Fix $x \in \O$.
	By Theorem \ref{Th: Fractional Fundamental Theorem of Calculus} and Lemma \ref{Lemma: bound of constant c_{n,s}},
	\begin{equation}\label{eq:upoint}
	\left| u (x) \right| \leq C(n) \left[ \int_{\O_1} \frac{|D^s u(y)|}{|x-y|^{n-s}} \, dy + \int_{\O_1^c} \frac{|D^s u(y)|}{|x-y|^{n-s}} \, dy \right] .
	\end{equation}
	Now $\O_1 \subset B (x, 3 R)$, so
	\begin{equation}\label{eq:x-yn-s}
	\int_{\O_1} \frac{1}{|x-y|^{n-s}} \, dy \leq \int_{B(x, 3 R)} \frac{1}{|x-y|^{n-s}} \, dy = \frac{\sigma_{n-1}}{s} \left( 3 R \right)^s \leq C(n) \frac{1}{s} R .
	\end{equation}
	Similarly, $\O \subset B (y, 3 R)$ for every $y \in \O_1$, so
	\begin{equation}\label{eq:x-yn-sO}
	\int_{\O} \frac{1}{|x-y|^{n-s}} \, dx \leq C(n) \frac{1}{s} R .
	\end{equation}
	By \eqref{eq:x-yn-s} and H\"older's inequality,
	\[
	\int_{\O_1} \frac{|D^s u(y)|}{|x-y|^{n-s}} \, dy \leq \left[ C(n) \frac{1}{s} R \right]^{\frac{1}{p'}} \left( \int_{\O_1} \frac{|D^s u(y)|^p }{|x-y|^{n-s}} \, dy \right)^{\frac{1}{p}} .
	\]
	Therefore, using \eqref{eq:x-yn-sO}, we find
	\begin{equation}\label{eq:DsuOO1}
	\begin{split}
	\left[ \int_{\O} \left( \int_{\O_1} \frac{|D^s u(y)|}{|x-y|^{n-s}} \, dy \right)^p dx \right]^{\frac{1}{p}} & \leq \left[ C(n) \frac{1}{s} R \right]^{\frac{1}{p'}} \left( \int_{\O_1} |D^s u (y)|^p \int_{\O} \frac{1}{|x-y|^{n-s}} \, dx \, dy \right)^{\frac{1}{p}} \\
	& \leq C(n) \frac{1}{s} R \left\| D^s u \right\|_p .
	\end{split}
	\end{equation}
	
	Now, for any $y \in \O_1^c$, by Lemma \ref{Lemma: bound of constant c_{n,s}},
	\begin{equation}\label{eq:Dsuy}
	\left| D^s u (y) \right| \leq C(n) \int \frac{\left| u(y) - u(z) \right|}{\left| y - z \right|^{n+s}} \, dz = C(n) \int_{\O} \frac{\left| u(z) \right|}{\left| y - z \right|^{n+s}} \, dz .
	\end{equation}
	When $z \in \O$ we have
	\[
	|y| \leq |y-z| + |z| \leq |y-z| + R \leq |y-z| + \frac{1}{2} |y| ,
	\]
	so $\frac{1}{2} |y| \leq |y-z|$ and, hence,
	\begin{equation}\label{eq:y-z}
	\frac{1}{|y-z|^{n+s}} \leq \left( \frac{2}{|y|} \right)^{n+s} \leq C(n) \frac{1}{|y|^{n+s}} .
	\end{equation}
	Similarly, for each $x \in \O$ we have
	\begin{equation}\label{eq:x-y}
	\frac{1}{|x-y|^{n-s}} \leq C(n) \frac{1}{|y|^{n-s}} .
	\end{equation}
	Using \eqref{eq:y-z} we find that
	\[
	\int_{\O} \frac{\left| u(z) \right|}{\left| y - z \right|^{n+s}} \, dz \leq C(n) \frac{1}{|y|^{n+s}} \left\| u \right\|_{L^1 (\O)} \leq C(n) |\O|^{\frac{1}{p'}} \frac{1}{|y|^{n+s}} \left\| u \right\|_{L^p (\O)} ,
	\]
	whence we infer from \eqref{eq:Dsuy} that
	\begin{equation} \label{eq: fr gradient bounded by Lp norm of u}
	\left| D^s u (y) \right| \leq C(n) |\O|^{\frac{1}{p'}} \frac{1}{|y|^{n+s}} \left\| u \right\|_{L^p (\O)} .
	\end{equation}
	Thus, using \eqref{eq:x-y} as well,
	\[
	\int_{\O_1^c} \frac{|D^s u(y)|}{|x-y|^{n-s}} \, dy \leq C(n) |\O|^{\frac{1}{p'}} \left\| u \right\|_{L^p (\O)} \int_{\O_1^c} \frac{1}{|y|^{n+s}} \frac{1}{|y|^{n-s}} \, dy = C(n) |\O|^{\frac{1}{p'}} R^{-n} \left\| u \right\|_{L^p (\O)} .
	\]
	This last inequality, combined with \eqref{eq:upoint} and \eqref{eq:DsuOO1}, implies by the triangular inequality that
	\begin{align*}
	\left\| u \right\|_{L^p (\O)} & \leq C(n) \frac{1}{s} R \left\| D^s u \right\|_p + C_1 (n) |\O|^{\frac{2}{p'}} R^{-n} \left\| u \right\|_{L^p (\O)} \\
	& \leq C(n) \frac{1}{s} R \left\| D^s u \right\|_p + C_1 (n) \max \{ 1, |\O|^2 \} R^{-n} \left\| u \right\|_{L^p (\O)} .
	\end{align*}
	Finally, we choose $R$ such that, in addition to \eqref{eq:R}, satisfies $C_1 (n) \max \{ 1, |\O|^2 \} R^{-n} \leq \frac{1}{2}$, so that $R$ depends on $n$ and $\O$.
	We obtain that
	\[
	\frac{1}{2} \left\| u \right\|_{L^p (\O)} \leq C(n) \frac{1}{s} R \left\| D^s u \right\|_p
	\]
	and concludes the proof.
\end{proof}

We will use the following immediate consequence of Theorem \ref{th:Poincare}.

\begin{cor} \label{cor: fractional Sobolev inequality definitiva}
	Let $\O \subset \Rn$ be open and bounded, and let $0 < s_0 < 1$.
	Then there exists $C = C (n,\O, s_0)$ such that for all $s_0 < s < 1$, $1 < p < \infty$ and
	$u \in H^{s,p}_0 (\O)$,
	\[
	\left\| u \right\|_{L^p (\O)} \leq C \left\| D^s u \right\|_p .
	\]
\end{cor}

 \section[Localization of fractional gradients]{Localization of fractional gradients}\label{se:localization}

In this section we prove the convergence of the $s$-fractional gradient of a $W^{1,p}$ function to its local gradient as $s \nearrow 1$. This result is to be expected, and easy to obtain for smooth functions using the Fourier transform (see Lemma \ref{Lemma: Fourier transform of the fractional gradient}).
In this section we provide a complete proof for functions in $W^{1,p}(\Rn)$.
This result, which is of interest in its own right, is a first step to prove the $\Gamma$-convergence of the functional $\mathcal{I}_s$ to $\mathcal{I}$ (see the Introduction).
It should be compared with \cite[Cor.\ 2]{BoBrMi2001}, where the convergence of the Gagliardo seminorm to the $L^p$ norm of the fractional gradient is shown (see also \cite[Prop.\ 15.7]{ponce_book}). 

% \begin{cor}\label{Cor: inequality with W1p}
% 	Let $p \geq 1$ and $u \in W^{1,p}(\mathbb{R}^n)$. Then, for every $s_0>0$ there exists a constant $C>0$ such that for every $s$, $s_0\leq s <1 $ we have 
% 	\begin{equation*}
% 	\lVert D^su\rVert_p \leq C\lVert u \rVert_{W^{1,p}(\Rn)}.
% 	\end{equation*}
% \end{cor}

We first recall the definition of Riesz potential, since it will be used in this section to relate the fractional gradient to the classical gradient.
Given $0<s<n$, the Riesz kernel $I_s : \R^n \setminus \{0\} \to \R$ is
\begin{equation*}
I_s (x)= \frac{1}{\gamma(s)}\frac{1}{|x|^{n-s}},
\end{equation*} 
where the constant $\gamma(s)$ is given by
\[
\gamma(s)=\frac{\pi^{\frac{n}{2}} \, 2^s \, \Gamma(\frac{s}{2})}{\Gamma(\frac{n-s}{2})} .
\]
The Riesz potential of a locally integrable function $f$ is given by
\begin{equation*}% \label{eq: Convolution with Riesz potential}
I_s * f(x)=\frac{1}{\gamma(s)} \int \frac{f(y)}{|x-y|^{n-s}}dy.
\end{equation*}
Note the relationship between $\gamma$ and $c_{n,s}$:
\begin{equation} \label{eq: constants relationship}
c_{n,s} = \frac{n+s-1}{\gamma (1-s)}.
\end{equation}

It is interesting to regard the $s$-fractional gradient from a Fourier analysis perspective.
As usual, the Fourier transform of a function $f:\Rn\to \mathbb{C}$ is defined as 
\[\hat{f}(\xi)=\int_{\Rn} f(x)e^{-2\pi i x\cdot \xi}\,dx.\]
We know that classical differentiation translates, when applying the Fourier transform, into multiplication of the Fourier transform of a function by a monomial.
This also happens in a fractional sense in this situation. The following result was proved in \cite[Th.\ 1.4]{ShS2015} (but it was mistakenly written with a sign switch); we include here a proof for the reader's convenience. 
%In what follows, we indicate by $\hat{}$ the Fourier transform.
%\textcolor{red}{Poner definici\'on, ya que no hay consenso sobre la constante multiplicativa}

\begin{lem} \label{Lemma: Fourier transform of the fractional gradient}
	Let $0<s<1$. Then, for all $u \in C_c^{\infty}(\Rn)$,
	\begin{equation*}
	\widehat{D^su} (\xi) = \frac{2 \pi i \xi }{|2 \pi \xi|^{1-s}}\hat{u} (\xi) , \qquad \xi \in \Rn .
	\end{equation*}
\end{lem}
\begin{proof} By \cite[Th.\ 1.2]{ShS2015}, $D^su=I_{1-s}*Du$ for any $u \in C_c^{\infty}(\Rn)$. We compute the Fourier transform of $D^s u$ in the sense of distributions. We start by checking that $I_{1-s} \in \mathcal{S}'$, where $\mathcal{S}$ is the Schwartz space.
	Given $\phi \in \mathcal{S}$,
	\begin{align*}
	\gamma (1-s) \langle I_{1-s}, \phi \rangle &= \int \frac{\phi(x)}{|x|^{n+s-1}}dx=\int_{B(0,1)}\frac{\phi(x)}{|x|^{n+s-1}}dx + \int_{B(0,1)^c} \phi(x)|x| \frac{1}{|x|^{n+s}}dx \\
	&\leq \lVert \phi\rVert_{\infty} \left\| \frac{1}{|x|^{n+s-1}} \right\|_{L^1(B(0,1))}+ \left\|\phi |x|\right\|_{\infty} \left\| \frac{1}{|x|^{n+s}} \right\|_{L^1(B(0,1)^c)} \\
	%&\leq P_{0,0}(\phi)\left\| \frac{1}{|x|^{n+s-1}} \right\|_{L^1(B(0,1))}+ P_{1,0}(\phi)\left\| \frac{1}{|x|^{n+s}} \right\|_{L^1(B(0,1)^c)} \\
	&\leq [ \lVert \phi\rVert_{\infty}+\left\|\phi |x|\right\|_{\infty}]\left[\left\| \frac{1}{|x|^{n+s-1}} \right\|_{L^1(B(0,1))}+ \left\| \frac{1}{|x|^{n+s}} \right\|_{L^1(B(0,1)^c)} \right],
	\end{align*}
	which shows that the Riesz potential is a continuous linear map over the Schwartz space.

	% 	 
	%  \eqref{eq: Convolution with Riesz potential},  where $\gamma(1-s)=\frac{-c_{n,s}}{n+s-1}$, we have 
	% 	\[
	% 	I_{1-s}*\varphi=\frac{c_{n,s}}{-(n+s-1)}\int_{\mathbb{R}^n}\frac{\varphi(y)}{|x-y|^{n-(1-s)}}dy.
	% 	\]

	%Secondly, as $Du \in W^{1,p}$, we have that $Du \in \mathcal{S}'$. Therefore, we can take the convolution of this two tempered distributions, which is a continuous map, $I_{1-s} *Du$.
	Now, $D^s u \in L^1(\Rn)$, since $u \in C^{\infty}_c (\Rn)$ (see, e.g., \cite[Lemma 3.1]{BeCuMC}) and, so, $D^s u$ can also be regarded as a tempered distribution. Therefore, we can apply the Fourier transform to $D^su=I_{1-s}*Du$ and, having in mind that the latter is a convolution of the Riesz potential with a Schwartz function, as well as that $\widehat{I_{1-s}} (\xi)=|2\pi \xi|^{-(1-s)}$ (see \cite{Stein70}), we have
	\begin{equation*}
	\widehat{D^su} (\xi) =\widehat{I_{1-s}*Du} (\xi) =\widehat{I_{1-s}} (\xi) \, \widehat{Du} (\xi) =|2\pi \xi|^{-(1-s)} \widehat{Du} (\xi) =\frac{2\pi i\xi }{|2\pi \xi|^{1-s}} \hat{u} (\xi) , \qquad \xi \in \R^n , %=|2\pi \xi|^s \hat{u}. 
	\end{equation*}
	as desired.
\end{proof}

The main result of the section is the following.   
\begin{teo} \label{Prop: convergencia del gradiente fraccionario al clásico}
	Let $0<s<1$ and $1<p<\infty$. Then, for each $u\in W^{1,p}(\Rn)$,
	\begin{equation*}
	D^su \rightarrow D u \text{ in $L^p(\mathbb{R}^n)$ as } s \nearrow 1 .
	\end{equation*}
\end{teo}
\begin{proof} We first prove the result for smooth functions and then extend it by density to $W^{1,p}(\R^n)$.
	
	%	
	%	 	The outline of the proof is the following. We first assume $u \in C_c^{\infty}(\Rn)$. We will compute the Fourier transform (in the sense of distributions) of the $s$-fractional gradient. Once there, we will notice that the Fourier transforms are actually in $L^1(\Rn)$ for every $s$. Consequently we will be able to apply the continuity of the Fourier transform from $L^1(\Rn)$ to $L^{\infty}(\Rn)$, obtaining uniform continuity of the sequence of fractional gradients of $C_c^{\infty}(\Rn)$ functions. The penultimate step will be to use an inequality between the $p$ and infinity norms in order to get strong convergence in $L^p(\Rn)$. Lastly, we will extend such result for $W^{1,p}(\Rn)$ functions. \\
	
	Let $u \in C_c^\infty (\Rn)$. By Lemma \ref{Lemma: Fourier transform of the fractional gradient},
	\begin{equation*} %\label{eq: Fourier transform os fractional gradient, sequence}
	\widehat{D^su} (\xi) = \frac{2\pi i\xi }{|2\pi \xi|^{1-s}}\hat{u} (\xi) , \qquad \xi \in \Rn , %=|2\pi \xi|^s \hat{u}. 
	\end{equation*}
	so by the elementary inequality $t^s \leq 1 + t$ for all $t \geq 0$,
	\begin{equation}\label{eq:Dsbound}
	\left| \widehat{D^su} (\xi) \right| = |2\pi \xi|^s \left| \hat{u} (\xi) \right| \leq \left( 1 + \left| 2\pi \xi \right| \right) \left| \hat{u} (\xi) \right| .
	\end{equation}
	As $\hat{u}$ is in the Schwartz space (because $u \in C^{\infty}_c (\Rn)$), both $\hat{u}$ and $\xi \, \hat{u} (\xi)$ are in $L^1 (\Rn)$.
	Therefore, $\widehat{D^su} \in L^1 (\Rn)$.
	On the other hand, by basic properties of the Fourier transform, $\widehat{Du} (\xi) = 2 \pi i \xi\hat{u} (\xi)$, so clearly, $\widehat{D^su} \to \widehat{Du}$ a.e.\ as $s \nearrow 1$.
	Thanks to the bound \eqref{eq:Dsbound} and dominated convergence, $\widehat{D^su} \to \widehat{Du}$ in $L^1(\R^n)$.
	As the inverse Fourier transform is continuous from $L^1(\Rn)$  to $L^{\infty}(\Rn)$, we also have that 
	\[
	D^s u\rightarrow Du \ \text{  uniformly in } \Rn.
	\]
	Now, using a standard interpolation inequality (or H\"older's), we get that
	\begin{align*}%\label{eq: interpolation of norms}
	%\begin{split}
	\lVert D^su- Du\rVert_p & \leq \lVert D^s u -Du\rVert_1^{\frac{1}{p}} \lVert D^su -Du \rVert_{\infty}^{\frac{1}{p'}} \\
	&\leq \left( \lVert D^su \rVert_1 + \lVert Du \rVert_1 \right)^{\frac{1}{p}} \lVert D^su-Du \rVert_{\infty}^{\frac{1}{p'}} \\
	&\leq C \lVert u\rVert_{W^{1,1}(\mathbb{R}^n)}^{\frac{1}{p}} \lVert D^su -Du \rVert_{\infty}^{\frac{1}{p'}},
	%\end{split}
	\end{align*}
	where we have used Proposition \ref{Lemma: inequality with W1p}, considering that as $s \nearrow 1$ we can assume $s \geq \frac{1}{2}$, so the constant $C>0$ does not depend on $s$. Thus, the convergence $D^s u\to Du$ in $L^p$ follows and the result is true for $C^{\infty}_c$ functions.
	
	To conclude the proof, we extend this result through a density argument. 
	Let us consider $u \in W^{1,p}(\Rn)$.
	Then, for every $\varepsilon>0$ we can find $v \in C_c^{\infty} (\Rn)$ such that $\lVert v-u\rVert_{W^{1,p}(\mathbb{R}^n)}< \varepsilon$.
	Thus,
	\begin{align*}
	\lVert D^su - Du \rVert_p &\leq  \lVert D^su - D^sv\rVert_p+ \lVert D^sv - Dv \rVert_p+ \lVert Dv - Du \rVert_p \\
	& \leq (C+1)\varepsilon +\lVert D^sv - Dv \rVert_p,
	\end{align*}
	where we have used  again Proposition \ref{Lemma: inequality with W1p}. Finally, when we take limits we obtain that
	\begin{equation*}
	\limsup_{s \nearrow 1}\lVert D^su - Du \rVert_p \leq (C+1)\varepsilon,
	\end{equation*}
	for every $\varepsilon>0$, which concludes the result.
\end{proof}

Thanks to Lemma \ref{le:div}, the previous result also implies the convergence in $L^p$ of the fractional divergence. 

\begin{cor} \label{Cor: strong convergence of the fractional divergence}
	Let $0<s<1$ and $1<p<\infty$. Then, for each $\phi \in W^{1,p}(\Rn, \Rn)$,
	\begin{equation*}
	\diver^s \phi \rightarrow \diver \phi \text{ in $L^p(\mathbb{R}^n)$ as } s \nearrow 1 .
	\end{equation*}
\end{cor}

\section[Compactness]{Compactness}\label{se:compactness}

In this section we establish that any sequence $\{ u_s \}_{s \in (0,1)}$ with bounded $H^{s,p}_g(\O)$ norm is precompact in $L^q(\Rn)$ for a suitable $q\ge 1$. 

Even though the continuous embedding of $H^{s,p}$ into $H^{\bar{s},p}$ for $0<\bar{s}<s<1$ is already known, we start by giving a new proof of this result, where we show that the embedding constant is independent of $s$. This proof follows the ideas of Theorem \ref{th:Poincare}.

\begin{prop} \label{Prop: franctional gradients ineq}
Let $0< \bar{s} < s_0 <1$. Let $\Omega\subset \mathbb{R}^n$ be a bounded open set. 
	Then, there exists a constant $C=C(\O ,n,s_0,\bar{s})>0$ such that for every $s \in [s_0 ,1)$, $1<p<\infty$ and $u \in H_0^{s,p}(\O)$ we have 
	\begin{equation} \label{ineq: embbedding of Bessel spaces}
	\|D^{\bar{s}} u\|_p \leq C \|D^s u\|_{p} .
	\end{equation} 
\end{prop}
\begin{proof}
By density, it is enough to prove the inequality for $u \in C^{\infty}_c (\O)$.
We divide the proof into two steps.

\smallskip

	\emph{Step 1.} First, we prove that there exists $C = C(\O, n, s_0, \bar{s}) > 0$ such that 
	\begin{equation} \label{ineq: fractional gradients inequality}
	\|D^{\bar{s}} u\|_{L^p(\O)} \leq C \|D^s u\|_p .
	\end{equation}
	Let $R \geq 1$ be such that $\O \subset B (0, R)$.
	Define $\O_1 := B (0, 2 R)$ and fix $x \in \O$.
Notice that, as a consequence of \cite[Th.\ 1.2]{ShS2015} and the semigroup property of the Riesz potential, we can write
\begin{equation*}%\label{eq: fractional gradients relation}
D^{\bar{s}}u=I_{1-\bar{s}}*Du= (I_{1-s}*I_{s-\bar{s}})*Du=I_{s-\bar{s}}* D^su .
\end{equation*}
This equality, together with \eqref{eq: constants relationship} and Lemma \ref{Lemma: bound of constant c_{n,s}}
yields
	\begin{equation}\label{eq:upoint 2}
	\begin{split}
	\left| D^{\bar{s} }u (x) \right| &\leq \frac{1}{\gamma(s-\bar{s})} \left[ \int_{\O_1} \frac{|D^s u(y)|}{|x-y|^{n-(s-\bar{s})}} \, dy + \int_{\O_1^c} \frac{|D^s u(y)|}{|x-y|^{n-(s-\bar{s})}} \, dy \right] \\
	&\leq C(n) \left[ \int_{\O_1} \frac{|D^s u(y)|}{|x-y|^{n-(s-\bar{s})}} \, dy + \int_{\O_1^c} \frac{|D^s u(y)|}{|x-y|^{n-(s-\bar{s})}} \, dy \right] .
	\end{split}
	\end{equation}
	Now $\O_1 \subset B (x, 3 R)$, so
	\begin{equation}\label{eq:x-yn-s 2}
	\int_{\O_1} \frac{1}{|x-y|^{n-(s-\bar{s})}} \, dy \leq \int_{B(x, 3 R)} \frac{1}{|x-y|^{n-(s-\bar{s})}} \, dy = \frac{\sigma_{n-1}}{s-\bar{s}} \left( 3 R \right)^{(s-\bar{s})} \leq C(n) \frac{1}{s-\bar{s}} R .
	\end{equation}
	Similarly, $\O \subset B (y, 3 R)$ for every $y \in \O_1$, so
	\begin{equation}\label{eq:x-yn-sO 2}
	\int_{\O} \frac{1}{|x-y|^{n-(s-\bar{s})}} \, dx \leq C(n) \frac{1}{s-\bar{s}} R .
	\end{equation}
	By \eqref{eq:x-yn-s 2} and H\"older's inequality,
	\[
	\int_{\O_1} \frac{|D^s u(y)|}{|x-y|^{n-(s-\bar{s})}} \, dy \leq \left[ C(n) \frac{1}{s-\bar{s}} R \right]^{\frac{1}{p'}} \left( \int_{\O_1} \frac{|D^s u(y)|^p }{|x-y|^{n-(s-\bar{s})}} \, dy \right)^{\frac{1}{p}} .
	\]
	Therefore, using \eqref{eq:x-yn-sO 2} and Fubini's theorem, we find, as in \eqref{eq:DsuOO1},
	\begin{equation}\label{eq:DsuOO1 2}
	\left[ \int_{\O} \left( \int_{\O_1} \frac{|D^s u(y)|}{|x-y|^{n-(s-\bar{s})}} \, dy \right)^p dx \right]^{\frac{1}{p}} \leq C(n) \frac{1}{s-\bar{s}} R \left\| D^s u \right\|_p .
	\end{equation}
	
	Now, for any $y \in \O_1^c$, % by Lemma \ref{Lemma: bound of constant c_{n,s}},
	%	\begin{equation}\label{eq:Dsuy 2}
	%	\left| D^s u (y) \right| \leq C(n) \int \frac{\left| u(y) - u(z) \right|}{\left| y - z \right|^{n+s}} \, dz = C(n) \int_{\O} \frac{\left| u(z) \right|}{\left| y - z \right|^{n+s}} \, dz .
	%	\end{equation}
	%	When $z \in \O$ we have
	%	\[
	%	|y| \leq |y-z| + |z| \leq |y-z| + R \leq |y-z| + \frac{1}{2} |y| ,
	%	\]
	%	so $\frac{1}{2} |y| \leq |y-z|$ and, hence,
	%	\begin{equation}\label{eq:y-z 2}
	%	\frac{1}{|y-z|^{n+s}} \leq \left( \frac{2}{|y|} \right)^{n+s} \leq C(n) \frac{1}{|y|^{n+s}} .
	%	\end{equation}
	similarly to \eqref{eq:x-y}, for each $x \in \O$ we have
	\begin{equation}\label{eq:x-y 2}
	\frac{1}{|x-y|^{n-(s-\bar{s})}} \leq C(n) \frac{1}{|y|^{n-(s-\bar{s})}} 
	\end{equation}
and, in fact, \eqref{eq: fr gradient bounded by Lp norm of u} also holds.
	%	Using \eqref{eq:y-z 2} we find that
	%	\[
	%	\int_{\O} \frac{\left| u(z) \right|}{\left| y - z \right|^{n+s}} \, dz \leq C(n) \frac{1}{|y|^{n+s}} \left\| u \right\|_{L^1 (\O)} \leq C(n) |\O|^{\frac{1}{p'}} \frac{1}{|y|^{n+s}} \left\| u \right\|_{L^p (\O)} ,
	%	\]
	%	whence we infer from \eqref{eq:Dsuy 2} that
	%	\[
	%	\left| D^s u (y) \right| \leq C(n) |\O|^{\frac{1}{p'}} \frac{1}{|y|^{n+s}} \left\| u \right\|_{L^p (\O)} .
	%	\]
Thus, using \eqref{eq:x-y 2} and \eqref{eq: fr gradient bounded by Lp norm of u},
	\[
	\int_{\O_1^c} \frac{|D^s u(y)|}{|x-y|^{n-(s-\bar{s})}} \, dy \leq C(n) |\O|^{\frac{1}{p'}} \left\| u \right\|_{L^p (\O)} \int_{\O_1^c} \frac{1}{|y|^{n+s}} \frac{1}{|y|^{n-(s-\bar{s})}} \, dy = C(n) |\O|^{\frac{1}{p'}} \frac{R^{-n-\bar{s}}}{n+\bar{s}} \left\| u \right\|_{L^p (\O)} .
	\]
	This last inequality, combined with \eqref{eq:upoint 2} and \eqref{eq:DsuOO1 2}, implies by the triangular inequality that
	\begin{equation*}
	\left\| D^{\bar{s}}u \right\|_{L^p (\O)} \leq C (n) \frac{1}{s - \bar{s}} R \| D^s u \|_p + |\O| C(n) \frac{R^{-n-\bar{s}}}{n+ \bar{s}} \|u \|_{L^p (\O)} .
	%& \leq C(n) \frac{1}{s} R \left\| D^s u \right\|_p + C_1 (n) \max \{ 1, |\O|^2 \} R^{-n} \left\| u \right\|_{L^p (\O)} .
	\end{equation*}
%\[
% \leq C\frac{1}{s_0-\bar{s}} R \left\| D^s u \right\|_p + C \left\| u \right\|_{L^p (\O)}
%\]
	Finally, we apply Theorem \ref{th:Poincare} on the right hand side to obtain
\[
 \left\| D^{\bar{s}}u \right\|_{L^p (\O)} \leq C (n, \O) \left( \frac{1}{s - \bar{s}} + \frac{R^{-n-\bar{s}}}{n+ \bar{s}} \frac{1}{s} \right) \| D^s u \|_p \leq C (n, \O) \left( \frac{1}{s_0 - \bar{s}} + \frac{1}{s_0 (n+ \bar{s})} \right) \| D^s u \|_p ,
\]
which completes the proof of \eqref{ineq: fractional gradients inequality}. 
	
\smallskip
	
	\emph{Step 2.} Now we prove \eqref{ineq: embbedding of Bessel spaces}. Let us call $\O_C=\O+B(0,1)$.
Then,
\[
 \left\| D^{\bar{s}} u \right\|_p \leq \left\| D^{\bar{s}} u \right\|_{L^p (\O_C)} + \left\| D^{\bar{s}} u \right\|_{L^p(\O_C^c)} .
\]
By \eqref{ineq: fractional gradients inequality} there exists $C>0$ (depending on $\O_C, n, s_0, \bar{s}$, so, ultimately, on $\O, n, s_0, \bar{s}$) such that
	\begin{equation} \label{eq: corollary fractional gradients ineq 1}
	\|D^{\bar{s}}u\|_p \leq C \|D^s u\|_p + \|D^{\bar{s}} u\|_{L^p(\O_C^c)}.
	\end{equation}
Now, for $x \in \O_C^c$,
\[
 D^{\bar{s}} u (x) = - c_{n, \bar{s}} \int_{\O} \frac{u(y)}{|x-y|^{n+ \bar{s}}} \frac{x-y}{|x-y|} dy ,
\]
so, by Lemma \ref{Lemma: bound of constant c_{n,s}},
\[
 \left| D^{\bar{s}} u (x) \right| \leq C(n) \int_{\O} \frac{|u(y)|}{|x-y|^{n+ \bar{s}}} dy ,
\]
and, hence, by Minkowski's integral inequality,
	\begin{equation} \label{eq: corollary fractional gradients ineq 2}
 \left\| D^{\bar{s}} u \right\|_{L^p (\O_C^c)} \leq C(n) \left( \int_{\O_C^c} \left( \int_{\O} \frac{|u(y)|}{|x-y|^{n+ \bar{s}}} dy \right)^p dx \right)^{\frac{1}{p}} \leq C(n) \int_{\O} |u(y)| \left( \int_{\O_C^c} \frac{1}{|x-y|^{(n+ \bar{s})p}} dx \right)^{\frac{1}{p}} dy ,
	\end{equation}
Now, for every $y \in \O$ we have $\O_C^c-y \subset B(0,1)^c$, and hence
	\begin{equation*}
	\int_{\O_C^c}  \frac{1}{|x-y|^{(n+\bar{s})p}} \,dx= 
	\int_{\O_C^c-y}  \frac{1}{|z|^{(n+\bar{s})p}} \,dx \leq 	\int_{B(0,1)^c}  \frac{1}{|z|^{(n+\bar{s})p}} \,dx =\frac{\sigma_{n-1}}{(n + \bar{s}) p - n} \leq \frac{\sigma_{n-1}}{\bar{s}} .
	\end{equation*}
	Thus, continuing from  \eqref{eq: corollary fractional gradients ineq 2} we find that
	\begin{equation} \label{eq: corollary fractional gradients ineq 3}
 \left\| D^{\bar{s}} u \right\|_{L^p (\O_C^c)} \leq C(n) \max \{ 1, \frac{\sigma_{n-1}}{\bar{s}} \} \int_{\O} |u(y)| dy \leq C(n) \max \{ 1, \frac{\sigma_{n-1}}{\bar{s}} \} \max \{ 1, |\O| \} \left\| u \right\|_{L^p (\O)} .
	\end{equation}
Inequalities \eqref{eq: corollary fractional gradients ineq 1}, \eqref{eq: corollary fractional gradients ineq 3} and Theorem \ref{th:Poincare} finish the proof.
\end{proof}

Now we present the main result of this section.
The proof of the following compactness result is partly inspired by that of \cite[Lemma 3.6]{MeD}.
This result should be compared with \cite[Th.\ 1.2]{Pon}, in which a $W^{s,p}$ version is done.
In what follows, given $p \in [1, n)$ we denote by $p^*$ its Sobolev conjugate exponent, i.e., $p^* = \frac{pn}{n-p}$.
Recall also the notation $p^*_s$ from Theorem \ref{Bessel embedding theorem}.

\begin{teo} \label{Pr: weak convergence in s}
	Let $1<p<\infty$ and $g\in W^{1,p}(\Rn)$.
For each $s \in (0, 1)$, let $u_s \in H_g^{s,p}(\Omega)$ be such that the family $\{D^s u_s\}_{s \in (0,1)}$ is bounded in $L^p(\Rn)$.
	Then, there exist $u \in W^{1,p} (\Rn)$ and an increasing sequence $\{ s_j \}_{j \in \mathbb{N}}\subset (0, 1)$ with $\lim_{j\to \infty} s_j = 1$ such that 
	for every $q$ satisfying
	\begin{equation*}%\label{ conditions of q}
	\begin{cases}
	q\in \left[p, p^* \right) & \text{if } p<n , \\
	q\in[p, \infty) & \text{if } p = n , \\
	q\in[p, \infty] & \text{if } p > n ,
	\end{cases}
	\end{equation*}
	there exists $j_q \in \mathbb{N}$ for which $\{u_{s_j}\}_{j\ge j_q}\subset L^q(\Rn)$ and the convergences 
	\[u_{s_j}  \rightarrow u \text{ in  } L^q(\Rn) \quad \text{and} \quad D^{s_j} u_{s_j} \rightharpoonup Du \text{ in } L^p(\Rn)\]
hold as $j \to\infty$.
\end{teo}
\begin{proof}
	Thanks to Theorem \ref{Prop: convergencia del gradiente fraccionario al clásico} and the Sobolev embedding, we can assume, without loss of generality, that $g=0$.
	
	Fix $0<\bar{s}<s_0<1$. By hypothesis and Proposition \ref{Prop: franctional gradients ineq}, $\{D^{\bar{s}} u_s\}_{s\in [s_0,1)}$ is bounded in $L^p(\Rn,\Rn)$, and consequently, by Corollary \ref{cor: fractional Sobolev inequality definitiva}, $\{u_s\}_{s\in[s_0,1)}$ is bounded in $H^{\bar{s},p}_0(\O)$. Since $H^{\bar{s},p}_0(\O)$ is reflexive, there exist $u\in H^{\bar{s},p}_0(\O)$ and an increasing sequence $\{s_j\}_{j\ge 1}\subset [s_0,1)$, with $\lim_{j\to \infty} s_j =1$, such that 
	\[u_{s_j}\weakc u\quad \text{in } H^{\bar{s},p}_0(\O).\]

Now, if $p\le n$, given $q\in [p,p^*)$ (defining $p^* = \infty$ if $p=n$), there exists $j_q\in\mathbb{N}$ such that for all $j \geq j_0$ we have $q< p^*_{s_j}$.
Arguing as above we obtain that $u_{s_j}\weakc u$ in $H^{s_{j_0},p}_0 (\O)$, so applying Theorem \ref{Bessel embedding theorem}, we have that $\{u_{s_j}\}_{j\ge j_q}\subset L^q(\Rn)$ and 
\[u_{s_j}\to u\quad \text{in }L^q(\Rn).\]
If $p>n$, there exists $j_0 \in \N$ such that $s_{j_0}p>n$, and arguing as above using  again Theorem \ref{Bessel embedding theorem}, we have that $\{u_{s_j}\}_{j\ge j_0}\subset L^q(\Rn)$, for any $q\in [p,+\infty]$, and 
\[u_{s_j}\to u\quad \text{in }L^q(\Rn).\]

	Next, as $\{D^{s_j} u_{s_j} \}_{j \geq j_0}$ is bounded in $L^p(\Rn, \Rn)$, there exists $V\in L^p(\Rn, \Rn)$ such that 
	$D^{s_j}u_{s_j}\rightharpoonup V$ in $L^p(\R^n, \Rn)$ as $j \to \infty$, in principle up to a subsequence, but we will see that in fact it holds true for the whole sequence.
Given $\varphi \in C_c^1(\mathbb{R}^n, \Rn)$, using the fractional integration by parts, Lemma \ref{le:divgrad}, we get
	\begin{equation*}
	\int D^{s_j} u_{s_j}(x) \cdot \varphi(x) \, dx = -\int u_{s_j}(x) \diver^{s_j} \varphi(x) \, dx,
	\end{equation*}
	and passing to the limit as $j \to \infty$, having in mind that both $u_{s_j}$ and $\diver^{s_j} \varphi$ are strongly convergent (Corollary \ref{Cor: strong convergence of the fractional divergence}), 
	we obtain
	\begin{equation*}
	\int V(x) \cdot \varphi(x)\,dx = -\int u(x) \diver \varphi(x) \, dx ,
	\end{equation*}
and hence $Du=V$ and $u\in W^{1,p}(\Rn)$.
Since this $V$ is unique, this shows that $D^{s_j} u_{s_j} \rightharpoonup V$ in $L^p(\R^n, \Rn)$ as $j \to \infty$ without the need of taking a subsequence.
	This finishes the proof.
\end{proof}

\section{Weak continuity of the minors for varying $s$}\label{sc:weakcontinuity}

In this section we prove the analogue in this context of the weak continuity of minors, namely, that if we have a sequence $\{ u_s \}_{s \in (0,1)}$ such that $u_s \in H^{s,p}$ for each $s$ and $D^s u_s \weakc D u$ in $L^p$ as $s \nearrow 1$ for some $u \in W^{1,p}$ then the minors of $D^s u_s$ converge weakly in some $L^q$ to the minors of $D u$.
For this, we follow the general guidelines of \cite{BeCuMC}, where the analogue convergence for a fixed $s$ is proved.
In essence, this section consists of an adaptation of many results of \cite{BeCuMC} with bounds that do not depend on $s$.

We start with an analogue of \cite[Lemma 3.1]{BeCuMC}.

\begin{lem} \label{Lema difference quotient bound}
	Let $\varphi \in C^{\infty}_c (\mathbb{R}^n)$ and let $\bar{\alpha} \in (0,1)$.
	Then
	\[
	\sup_{s \in (\bar{\alpha}, 1)} \sup_{x \in \Rn} c_{n,s} \int \frac{\left| \varphi(x)-\varphi(y) \right|}{|x-y|^{n+s}} dy <\infty \quad \text{and} \quad \sup_{s \in (\bar{\alpha}, 1)} \sup_{r \in [1,\infty]} \left\| D^s \varphi \right\|_r < \infty .
	\]
\end{lem}

\begin{proof}
	Fix $\alpha \in (0, \bar{\alpha})$.
	Let $L$ and $C$ be, respectively, the Lipschitz and $\alpha$-H\"older constants of $\varphi$.
	Then, for every $x \in \mathbb{R}^n$,
	\begin{equation} \label{eq:infty norm bound}
	\begin{split}
	\int \frac{\left| \varphi(x)-\varphi(y) \right|}{|x-y|^{n+s}} \, dy &\leq \int_{B(x,1)} \frac{L}{|x-y|^{n+s-1}} \, dy+ \int_{B(x,1)^c} \frac{C}{|x-y|^{n+s-\alpha}} \, dy \\
	& = \int_{B(0,1)} \frac{L}{|z|^{n+s-1}} \, dz + \int_{B(0,1)^c} \frac{C}{|z|^{n+s-\alpha}} \, dz = \frac{\sigma_{n-1} L}{1-s} + \frac{\sigma_{n-1} C}{s - \alpha}.
	\end{split}
	\end{equation}
	By Lemma \ref{Lemma: bound of constant c_{n,s}}, we find that
	\[
	\sup_{s \in (\bar{\alpha}, 1)} \sup_{x \in \Rn} c_{n,s} \int \frac{\left| \varphi(x)-\varphi(y) \right|}{|x-y|^{n+s}} dy <\infty ,
	\]
	and, as a consequence,
	\[
	\sup_{s \in (\bar{\alpha}, 1)} \left\| D^s\varphi \right\|_{\infty} < \infty .
	\]
	
	Denote by $F$ the support of $\varphi$.
	Then
	\begin{equation} \label{eq:integrability of smooth functions 0}
	\int\left| c_{n,s}\int \frac{\varphi(x)-\varphi(y)}{|x-y|^{n+s}}\frac{x-y}{|x-y|} dy\right| dx \leq |c_{n,s}| \left( A + B \right) ,
	\end{equation}
	where
	\[
	A :=\int \int_{F} \frac{\left| \varphi(x)-\varphi(y) \right|}{|x-y|^{n+s}} dy \, dx, \qquad
	B :=\int \int_{F^c} \frac{\left| \varphi(x)-\varphi(y) \right|}{|x-y|^{n+s}} dy \, dx .
	\]
	Now, we observe that, applying Fubini's Theorem and \eqref{eq:infty norm bound},
	\begin{equation} \label{eq:integrability of smooth functions 1}
	A = \int_{F}\int  \frac{\left| \varphi(x)-\varphi(y) \right|}{|x-y|^{n+s}}  dx \, dy \leq \left( \frac{\sigma_{n-1} L}{1-s} + \frac{\sigma_{n-1} C}{s - \alpha} \right) |F|,
	\end{equation}
	where $|F|$ denotes the measure of the set $F$.
	
	We notice that $\left| \varphi(x)-\varphi(y) \right|=0$ for every $(x,y) \in F^c \times F^c$. Therefore, applying again \eqref{eq:infty norm bound} we get
	\begin{equation}\label{eq:integrability of smooth functions 2}
	B =\int_{F}\int_{F^c} \frac{\left| \varphi(x)-\varphi(y) \right|}{|x-y|^{n+s}}  dy \, dx \leq \int_{F}\int \frac{\left| \varphi(x)-\varphi(y) \right|}{|x-y|^{n+s}}  dy \, dx \leq \left( \frac{\sigma_{n-1} L}{1-s} + \frac{\sigma_{n-1} C}{s - \alpha} \right) |F|.
	\end{equation}
	Putting together \eqref{eq:integrability of smooth functions 0}, \eqref{eq:integrability of smooth functions 1} and \eqref{eq:integrability of smooth functions 2} we have that
	\[
	\| D^s \varphi \|_1 \leq 2 c_{n,s}\left( \frac{\sigma_{n-1} L}{1-s} + \frac{\sigma_{n-1} C}{s - \alpha} \right) |F| .
	\]
	Applying now Lemma \ref{Lemma: bound of constant c_{n,s}} we infer that
	\[
	\sup_{s \in (\bar{\alpha}, 1)} \| D^s \varphi \|_1 < \infty .
	\]
	Finally, through a standard interpolation argument, we get that
	\[
	\sup_{s \in (\bar{\alpha}, 1)} \sup_{r \in [1,\infty]} \left\| D^s \varphi \right\|_r < \infty .
	\] 
\end{proof}

We now recall a nonlocal operator related to the fractional gradient introduced in \cite{BeCuMC}.
Moreover, we adapt \cite[Lemma 3.2]{BeCuMC} to show bounds of this operator independent of $s$.

\begin{lem} \label{Lema operador lineal} 
	Let $1\leq q < \infty$ and $0 < s < 1$.
	Let $\varphi \in C^{\infty}_c (\mathbb{R}^n)$, $k \in \N$ and $r \in [1,q]$.
	Then, the operator $K^s_{\varphi}: L^q(\mathbb{R}^n, \mathbb{R}^{k \times n}) \rightarrow L^r (\mathbb{R}^n, \mathbb{R}^k)$ defined as
	\begin{equation*}
	K^s_{\varphi}(U)(x)= c_{n,s} \int \frac{\varphi(x)-\varphi(y)}{|x-y|^{n+s}} U(y) \frac{x-y}{|x-y|} dy , \qquad \text{a.e. } x \in \mathbb{R}^n ,
	\end{equation*}
	is linear and bounded.
	Moreover, given $0 < \bar{\alpha} < 1$, there exists a constant $C=C(n,q,\bar{\alpha}, \varphi)$ such that for every $s \in (\bar{\alpha}, 1)$, every $r \in [1,q]$ and $U \in L^q(\Rn, \R^{k \times n})$,
	\begin{equation*}
	\| K^s_\varphi(U)\|_r \leq C \left\| U\right\|_q.
	\end{equation*}
\end{lem}
\begin{proof}
	The operator $K^s_{\varphi}$ is clearly linear.
	Let $U \in L^q(\mathbb{R}^n, \mathbb{R}^{k \times n})$.
	For all $x \in \mathbb{R}^n$ we have
	\[
	\left| K^s_{\varphi} (U) (x) \right| \leq |c_{n,s}| \int  \frac{\left| \varphi(x)-\varphi(y) \right|}{|x-y|^{n+s}} \left| U(y) \right| dy ,
	\]
	so
	\begin{equation}\label{eq:K1}
	\left| K^s_{\varphi} (U) (x) \right|^q \leq 2^{q-1}|c_{n,s}|^q \left( g(x) + h(x) \right) ,
	\end{equation}
	with
	\[
	g(x) := \left(\int_{B(x,1)} \frac{\left| \varphi(x)-\varphi(y) \right|}{|x-y|^{n+s}} \left| U(y) \right| dy \right)^q, \qquad
	h(x) :=\left(\int_{B(x,1)^c} \frac{\left| \varphi(x)-\varphi(y) \right|}{|x-y|^{n+s}} \left| U(y) \right| dy \right)^q.
	\]
	Fix $\alpha \in (0, \bar{\alpha})$.
	Let $L$ and $L_{\alpha}$ be the Lipschitz and $\alpha$-H\"older semi-norm, respectively, of $\varphi$.
	Then, applying H\"older's inequality, we get
	\begin{align*}
	g(x) & \leq L^q \left(\int_{B(x,1)} \frac{|U(y)|}{|x-y|^{n+s-1}}dy \right)^q = L^q \left(\int_{B(0,1)}  \frac{|U(x-z)|}{|z|^{n+s-1}}dz \right)^q \\
	& \leq L^q \int_{B(0,1)} \frac{|U(x-z)|^q}{|z|^{n+s-1}}dz \left(\int_{B(0,1)} \frac{1}{|z|^{n+s-1}}dz\right)^{q-1} \\
	&= L^q \left(\frac{\sigma_{n-1}}{1-s}\right)^{q-1} \int_{B(0,1)}  \frac{|U(x-z)|^q}{|z|^{n+s-1}}dz ,
	\end{align*}
	where $\sigma_{n-1}$ is the area of the unit sphere of $\Rn$.
	Integrating,
	\begin{equation}\label{eq:K2}
	\int  g(x)\,dx \leq L^q \left(\frac{\sigma_{n-1}}{1-s}\right)^{q-1} \int_{B(0,1)}\frac{1}{|z|^{n+s-1}} \int  |U(x-z)|^q dx \,dz = L^q \left(\frac{\sigma_{n-1}}{1-s}\right)^{q} \left\| U\right\|_q^q . 
	\end{equation}
	
	As for the term $h$, applying H\"older's inequality,
	\begin{align*}
	h(x) & \leq L_{\alpha}^q \left( \int_{B(x,1)^c} \frac{|U(y)|}{|x-y|^{n+s-\alpha}} dy \right)^q  \\
	%=\left( \int_{B(0,1)^c}\left| U(z-x) \right| C Md\mu \right)^q 
	& \leq L_{\alpha}^q \int_{B(0,1)^c}\frac{\left| U(x-z) \right|^q}{|z|^{n+s-\alpha}} dz \left(\int_{B(0,1)^c}\frac{1}{|z|^{n+s-\alpha}}dz\right)^{{q-1}} \\
	&=  L_{\alpha}^q \left(\frac{\sigma_{n-1}}{s-\a}\right)^{q-1} \int_{B(0,1)^c}\frac{\left| U(x-z) \right|^q}{|z|^{n+s-\alpha}} dz.
	\end{align*}
	Integrating,
	\begin{equation}\label{eq:K3}
	\int h(x)\,dx \leq L_{\alpha}^q \left(\frac{\sigma_{n-1}}{s-\a}\right)^{q-1} \int_{B(0,1)^c}\frac{1}{|z|^{n+s-\alpha}} \int | U(x-z)|^qdx  \,dz = L_{\alpha}^q \left(\frac{\sigma_{n-1}}{s-\a}\right)^q \left\| U\right\|_q^q .
	\end{equation}
	Putting together \eqref{eq:K1}, \eqref{eq:K2} and \eqref{eq:K3} we obtain
	\[
	\left\| K^s_{\varphi}(U)\right\|_q^{q} \leq 2^{q-1} |c_{n,s}|^q \left( L^q \left(\frac{\sigma_{n-1}}{1-s}\right)^{q}+L_{\alpha}^q \left(\frac{\sigma_{n-1}}{s-\a}\right)^q\right) \left\| U\right\|_q^q,
	\]
	so applying Lemma \ref{Lemma: bound of constant c_{n,s}} we find that
	\begin{equation}\label{eq:interpol1}
	\left\| K^s_{\varphi}(U)\right\|_q \leq C \left\| U\right\|_q 
	\end{equation}
	for some constant $C$ independent of $s \in (\bar{\alpha}, 1)$ and $U$.
	
	Next, we are going to check the boundedness of $K^s_\varphi:L^q(\mathbb{R}^n, \R^{k \times n}) \rightarrow L^1(\mathbb{R}^n, \R^k)$.
	Denote by $F$ the support of $\varphi$.
	Then
	\begin{equation}\label{eq:2K1}
	\int \left| K^s_{\varphi} (U) (x) \right| dx \leq |c_{n,s}| \left( A + B \right) ,
	\end{equation}
	where
	\[
	A :=\int \int_{F} \frac{\left| \varphi(x)-\varphi(y) \right|}{|x-y|^{n+s}} \left| U(y) \right| dy \, dx, \qquad
	B :=\int \int_{F^c} \frac{\left| \varphi(x)-\varphi(y) \right|}{|x-y|^{n+s}} \left| U(y) \right| dy \, dx .
	\]
	Now, we observe that, applying Fubini's Theorem, H\"older's inequality and Lemma \ref{Lema difference quotient bound} there exists $C_0>0$ independent of $s \in (\bar{\alpha}, 1)$ such that
	\begin{equation}\label{eq:2K2}
	A \leq \int_{F}\left| U(y) \right|\int  \frac{\left| \varphi(x)-\varphi(y) \right|}{|x-y|^{n+s}}  dx \, dy\leq C_0|F|^{\frac{1}{q'}}\left(\int_{F}\left| U(y)\right|^q dy\right)^{\frac{1}{q}} \leq C_0|F|^{\frac{1}{q'}} \left\| U\right\|_q .
	\end{equation}
	Since $\left| \varphi(x)-\varphi(y) \right|=0$ for every $(x,y) \in F^c \times F^c$, in view of H\"older's inequality and Lemma \ref{Lema difference quotient bound} we get
	\begin{align*}
	B &=\int_{F}\int_{F^c} \frac{\left| \varphi(x)-\varphi(y) \right|}{|x-y|^{n+s}} \left| U(y) \right| dy \, dx \\
	&\leq \int_{F}\left(\int_{F^c} \frac{\left| \varphi(x)-\varphi(y) \right|}{|x-y|^{n+s}}dy\right)^{\frac{1}{q'}}\left(\int_{F^c}\frac{\left| \varphi(x)-\varphi(y) \right|}{|x-y|^{n+s}}\left| U(y) \right|^q dy\right)^{\frac{1}{q}}dx \\
	&\leq C_0^{\frac{1}{q'}}  \int_{F}\left(\int_{F^c}\frac{\left| \varphi(x)-\varphi(y) \right|}{|x-y|^{n+s}}\left| U(y) \right|^q dy\right)^{\frac{1}{q}}dx.
	\end{align*}
	Using again H\"older's inequality, Lemma \ref{Lema difference quotient bound} and Fubini's Theorem, we obtain
	\begin{equation}\label{eq:2K3}
	\begin{split}
	B &\leq C_0^{\frac{1}{q'}}  \left(\int_{F}\int_{F^c}\frac{\left| \varphi(x)-\varphi(y) \right|}{|x-y|^{n+s}}\left| U(y) \right|^q dy \, dx\right)^{\frac{1}{q}} \left|F \right|^{\frac{1}{q'}} \\
	&= \left( C_0 |F| \right)^{\frac{1}{q'}}\left(\int_{F^c}\left| U(y) \right|^q \int_{F}\frac{\left| \varphi(x)-\varphi(y) \right|}{|x-y|^{n+s}} dx \, dy \right)^{\frac{1}{q}} \\
	&\leq
	\left( C_0 |F| \right)^{\frac{1}{q'}}C_0^{\frac{1}{q}}\left(\int_{F^c}\left| U(y) \right|^qdy \right)^{\frac{1}{q}} \leq C \left\| U\right\|_q ,
	\end{split}
	\end{equation}
	where  $C>0$ is a constant independent of $s$ and $U$. % and $|F|$ denotes the measure of $F$.
	Inequalities \eqref{eq:2K1}, \eqref{eq:2K2} and \eqref{eq:2K3} lead us to
	\begin{equation}\label{eq:interpol2}
	\left\| K^s_{\varphi}(U) \right\|_1 \leq C \left\| U \right\|_q,
	\end{equation}
	for some constant $C$ independent of $s \in (\bar{\alpha}, 1)$ and $U$.
	The conclusion of the theorem is obtained through an interpolation of inequalities \eqref{eq:interpol1} and \eqref{eq:interpol2}.
\end{proof}

The following result is the key to adapt the continuity of minors of \cite{BeCuMC} to our case.
It establishes the relationship between the operators $K^s_{\varphi}$ and $D \varphi$ when $s \nearrow 1$.

\begin{lem} \label{lem: weak covergence of the linear operator K}
	Let $p >1$ and $0<s<1$.
	Let $\varphi \in C_c^{\infty}(\Rn)$ and $w \in L^p (\Rn, \Rnn)$.
	Consider a family $\{ w_s \}_{s \in (0,1)}$ in $L^p (\Rn, \Rnn)$ such that $w_s \weakc w$ in $L^p (\Rn, \Rnn)$ as $s \nearrow 1$.
	Then, for all $r \in (1, p]$,
	\[
	K^s_\varphi (w_s) \weakc w \, D \varphi \quad \text{in } L^r (\Rn, \Rn) \text{ as } s \nearrow 1 .
	\]
\end{lem}
\begin{proof}
	Assume first that $w_s \in W^{1,p} (\Rn, \Rnn)$ for all $s \in (0,1)$.
	Fix two indexes $1 \leq i,j \leq n$, let $u_s$ be the $(i,j)$-th entry of $w_s$ and let $u$ be the $(i,j)$-th entry of $w$.
	Let $\theta \in C_c^{\infty}(\Rn)$.
	We apply the product formula of \cite[Lemma 3.4]{BeCuMC} and then Lemma \ref{le:divgrad} to obtain
	\[
	\int \theta \, K^s_\varphi (u_s I) = \int \theta \, D^s(\varphi u_s) - \int \theta \, \varphi \, D^s u_s = - \int \varphi \, u_s \diver^s \theta + \int u_s \diver^s (\theta \varphi) .
	\]
	Now we have from Corollary \ref{Cor: strong convergence of the fractional divergence} that $\diver^s \theta \to \diver \theta$ and $\diver^s (\theta\varphi) \to \diver^s (\theta\varphi)$ in $L^q(\Rn)$ for every $q \in (1, \infty)$ as $s \nearrow 1$.
	As $u_s \weakc u$ in $L^p (\Rn)$ we obtain
	\[
	\int \theta \, K^s_\varphi(u_s I) \rightarrow - \int \varphi \, u  \diver \theta + \int u \diver(\theta \varphi) = \int \theta \, u \, D \varphi .
	\]
	This shows that $K^s_\varphi (u_s I) \weakc u \, D \varphi$ in the sense of distributions.
	Now, by Lemma \ref{Lema operador lineal}, for every $r \in (1,p]$
	\[
	\left\| K^s_\varphi (u_s I) \right\|_r \leq C \left\| u_s \right\|_p \leq C_1
	\]
	for some $C, C_1 >0$ independent of $s$, which implies that $K^s_\varphi (u_s I) \weakc u \, D \varphi$ in $L^r (\Rn)$.
	
	Now, we remove the assumption $w_s \in W^{1,p} (\Rn)$.
	Fix $r \in (1, p]$.
	For each $s \in (0, 1)$, let $v_s \in W^{1,p} (\Rn)$ be such that $\left\| u_s - v_s \right\|_r \leq 1-s$.
	By Lemma \ref{Lema operador lineal},
	\[
	\| K^s_\varphi (u_s I) - K^s_\varphi (v_s I) \|_r = \| K^s_\varphi ( u_s I - v_s I) \|_r \leq C \|  u_s - v_s \|_p \to 0 ,
	\]
	which implies that $K^s_\varphi (u_s I) \weakc u \, D \varphi$ in $L^r (\Rn)$.
	In other words, for each $j \in \{1, \ldots, n\}$, the family of functions
	\[
	x \mapsto c_{n,s} \int \frac{\varphi(x)-\varphi(y)}{|x-y|^{n+s}} u_s (y) \frac{x_j - y_j}{|x-y|} dy
	\]
	converges weakly to $u (D \varphi)_j$ as $s \nearrow 1$.
	Therefore, for each $i \in \{1, \ldots, n\}$, the family of functions
	\[
	x \mapsto \left(  K^s_\varphi (w_s) \right)_{ij} (x) = \sum_{j=1}^n c_{n,s} \int \frac{\varphi(x)-\varphi(y)}{|x-y|^{n+s}} (w_s)_{ij}(y) \frac{x_j - y_j}{|x-y|} dy
	\]
	converges weakly to $\sum_{j=1}^n w_{ij} (D \varphi)_j = (w \, D \varphi)_i$.
	This concludes the proof.
\end{proof}

We now show a convenient notation for submatrices, which is taken from \cite[Def.\ 4.1]{BeCuMC}.

\begin{definicion}\label{de:submatrix}
	Let $k \in \N$ be with $1 \leq k \leq n$.
	Consider indices $1 \leq i_1 < \cdots < i_k \leq n$ and $1 \leq j_1 < \cdots < j_k \leq n$.
	\begin{enumerate}[a)]
		\item
		We define $M=M_{i_1 , \ldots , i_k; j_1 , \ldots , j_k} : \R^{n \times n} \to \R^{k \times k}$ as the map such that $M(F)$ is the submatrix of $F \in \R^{n \times n}$ formed by the rows $i_1 , \ldots , i_k$ and the columns $j_1 , \ldots , j_k$.
		
		\item
		We define $\bar{M} = \bar{M}_{i_1 , \ldots , i_k; j_1 , \ldots , j_k} : \R^{k \times k} \to \Rnn$ as the map such that $\bar{M}(F)$ is the matrix whose rows $i_1 , \ldots , i_k$ and columns $j_1 , \ldots , j_k$ coincide with those of $F$, whereas the rest of the entries are zero.
		
		%\item We define $N = N_{i_1 , \ldots , i_k} : \Rn \to \R^k$ as the map such that $N(v)$ is the subvector of $v \in \Rn$ formed by the entries $i_1 , \ldots , i_k$.
		
		%\item We define $\bar{N} = \bar{N}_{i_1 , \ldots , i_k} : \R^k \to \Rn$ as the map such that $\bar{N} (v)$ is the vector whose entries $i_1 , \ldots , i_k$ coincide with those of $v$, whereas the rest of the entries are zero.
		
		\item
		We define $\tilde{N} = \tilde{N}_{i_1 , \ldots , i_k} : \Rn \to \Rn$ as the map such that $\tilde{N} (v)$ is the vector whose entries $i_1 , \ldots , i_k$ coincide with the corresponding entries of $v$, whereas the rest of the entries are zero.
	\end{enumerate}
\end{definicion}

The following is the main result of this section, and shows the weak convergence of the minors of $D^s u_s$ to those of $D u$, whenever $u_s$ converges weakly to $u$.
Of course, by a minor we mean the determinant of a submatrix.
Its proof is an adaptation of \cite[Th.\ 5.2]{BeCuMC}.

\begin{teo}\label{th:wcontdet}
	Let $p \geq n-1$ and $0 < s < 1$.
	Let $g \in W^{1,p} (\Rn, \Rn)$ and $u \in W^{1,p} (\Rn, \Rn)$.
	Let $\{ u_s \}_{s \in (0,1)}$ be a family such that $u_s \in H^{s,p}_g (\O, \Rn)$ for each $s \in (0,1)$, while $u_s \rightarrow u$ in $L^p(\Rn, \Rn)$ and $D^s u_s \weakc D u$ in $L^p (\Rn, \Rnn)$ as $s \nearrow 1$.
	Then
	\begin{enumerate}[a)]
		\item\label{item:wcontdetA} If $k \in\N$ with $1 \leq k \leq n-2$ and $\mu$ is a minor of order $k$ then $\mu (D^s u_s) \weakc \mu (D u)$ in $L^{\frac{p}{k}} (\Rn)$ as $s \nearrow 1$.
		
		\item\label{item:wcontcof} If $\cof D^s u_s \weakc \vartheta$ in $L^q (\Rn, \Rnn)$  as $s \nearrow 1$ for some $q \in [1, \infty)$ and $\vartheta \in L^q (\Rn, \Rnn)$ then $\vartheta = \cof D u$.
		
		\item\label{item:wcontdet}
		Assume $\det D^s u_s \weakc \theta$ in $L^{\ell} (\Rn)$ as $s \nearrow 1$ for some $\ell \in [1, \infty)$ and some $\theta \in L^{\ell} (\Rn)$.
		If $p < n$ assume, in addition, that $\cof D^s u_s \weakc \cof D u$ in $L^q (\Rn, \Rnn)$ as $s \nearrow 1$ for some $q \in ( \frac{p^*}{p^* -1}, \infty)$.
		Then $\theta = \det D u$.
	\end{enumerate}
\end{teo}
\begin{proof}
	
	We will prove \emph{\ref{item:wcontdetA})} by induction on $k$.
	For $k=1$ there is nothing to prove.
	Assume it holds for some $k \leq n-3$ and let us prove it for $k+1$.
	Let $\mu$ be a minor of order $k+1$.
	In the notation of Definition \ref{de:submatrix}, $\mu (F) = \det M(F)$ for all $F \in \Rnn$, where $M = M_{i_1, \ldots , i_{k+1} ; j_1, \ldots , j_{k+1}}$ for some $1 \leq i_1 < \cdots < i_{k+1} \leq n$ and $1 \leq j_1 < \cdots < j_{k+1} \leq n$.
	Let $\varphi \in C^{\infty}_c (\Rn)$.
	By induction assumption, $\cof M(D^s u_s) \weakc \cof M(D u)$ in $L^{\frac{p}{k}} (\Rn, \R^{(k+1) \times (k+1)})$ as $s \nearrow 1$, so $\bar{M}(\cof M(D^s u_s)) \weakc \bar{M}(\cof M(D u))$ in $L^{\frac{p}{k}} (\Rn, \Rnn)$.
	By Lemma \ref{lem: weak covergence of the linear operator K}, $K^s_{\varphi} (\bar{M}(\cof M(D^s u_s))) \weakc \bar{M}(\cof M(Du)) \, D\varphi$ in $L^r (\Rn, \Rn)$ for every $r \in (1, \frac{p}{k}]$.
	By Theorem \ref{Pr: weak convergence in s}, $\tilde{N}(u_s) \to \tilde{N}(u)$ in $L^p (\Rn)$, so 
	\begin{equation}\label{eq:NK}
	\tilde{N}(u_s) \cdot  K^s_{\varphi} (\bar{M}(\cof M(D^su_s))) \weakc \tilde{N}(u) \cdot \left( \bar{M}(\cof M(D u)) \, D \varphi \right) \quad \text{in } L^1 (\Rn)
	\end{equation}
	since $\frac{k}{p} + \frac{1}{p} \leq 1$.
	Now, the nonlocal integration by parts for the determinant given in \cite[Lemma 5.1]{BeCuMC} as well as the classical (local) one state that
	\begin{equation}\label{eq:recordarpartes1}
	- \frac{1}{k} \int \tilde{N} (u_s) (x) \cdot K_{\varphi} (\bar{M} (\cof M(D^s u_s))) (x) \, dx = \int \det M(D^s u_s) (x) \, \varphi (x) \, dx 
	\end{equation}
	and
	\begin{equation}\label{eq:recordarpartes2}
	- \frac{1}{k} \int \tilde{N}(u) (x) \cdot \left( \bar{M}(\cof M(D u)) (x) \, D \varphi (x) \right) dx = \int \det M(D u(x)) \, \varphi (x) \, dx ,
	\end{equation}
respectively, so
	\begin{equation}\label{eq:detM}
	\int \det M(D^s u_s (x)) \, \varphi (x) \, dx \to \int \det M(D u(x)) \, \varphi (x) \, dx .
	\end{equation}
	This shows that $\det M (D^s u_s) \weakc \det M (D u)$ in the sense of distributions.
	As $\{ \det M (D^s u_s) \}_{s \in (0,1)}$ is bounded in $L^{\frac{p}{k+1}} (\Rn)$ and $p > k+1$, we have that $\det M (D^s u_s) \weakc \det M (D u)$ in $L^{\frac{p}{k+1}} (\Rn)$ as $s \nearrow 1$.
	
	The proof of \emph{\ref{item:wcontcof})} follows the lines of \emph{\ref{item:wcontdetA})}.
	Let $\mu$ be a minor of order $n-1$.
	As before, $\mu (F) = \det M(F)$ for all $F \in \Rnn$, where $M = M_{i_1, \ldots , i_{n-1} ; j_1, \ldots , j_{n-1}}$ for some  $1 \leq i_1 < \cdots < i_{n-1} \leq n$ and $1 \leq j_1 < \cdots < j_{n-1} \leq n$.
	Let $\phi \in C^{\infty}_c (\O)$.
	By part \emph{\ref{item:wcontdetA})}, $\cof M(D^s u_j) \weakc \cof M(D^s u)$ in $L^{\frac{p}{n-2}} (\Rn, \R^{(n-1) \times (n-1)})$, so 
	$\bar{M}(\cof M(D^s u_s)) \weakc \bar{M}(\cof M(D u))$ in $L^{\frac{p}{n-2}} (\Rn, \Rnn)$.
	By Lemma \ref{lem: weak covergence of the linear operator K}, $K^s_{\varphi} (\bar{M}(\cof M(D^s u_s))) \weakc \bar{M}(\cof M(Du)) \, D\varphi $ in $L^r (\Rn, \Rn)$ for every $r \in (1, \frac{p}{n-2}]$.
	By Theorem \ref{Pr: weak convergence in s}, $\tilde{N}(u_s) \to \tilde{N}(u)$ in $L^p (\Rn)$, so convergence \eqref{eq:NK} is also valid since $\frac{n-2}{p} + \frac{1}{p} \leq 1$.
	Again thanks to \eqref{eq:recordarpartes1}--\eqref{eq:recordarpartes2}, we conclude that convergence \eqref{eq:detM} holds.
	This shows that $\mu (D^s u_s) \weakc \mu (D u)$ in the sense of distributions.
	As this is true for every minor $\mu$ of order $n-1$, we obtain that $\cof D^s u_s \weakc \cof D u$ in the sense of distributions.
	Due to the assumption, $\vartheta = \cof D u$.
	
	We finally show part \emph{\ref{item:wcontdet})}.
	Let $\phi \in C^{\infty}_c (\O)$.
	Assume first $p<n$.
	By the assumption and Lemma \ref{lem: weak covergence of the linear operator K}, $K^s_{\varphi} (\cof D^s u_s) \weakc \cof D u \, D\varphi$ in $L^r (\Rn, \Rn)$ for every $r \in (1, q]$.
	By Theorem \ref{Pr: weak convergence in s}, $u_s \to u$ in $L^t (\Rn)$ for every $t \in [p, p^*)$, so 
	\begin{equation}\label{eq:ucofDu}
	u_s \cdot  K^s_{\varphi} (\cof D^s u_s) \weakc u \cdot \left( \cof D u \, D\varphi \right) \quad \text{in } L^1 (\Rn)
	\end{equation}
	since $\frac{1}{q} + \frac{1}{p^*} < 1$.
	
	Assume now $p \geq n$.
	Then $\{ \cof D^s u_s \}_{s \in (0,1)}$ is bounded in $L^{\frac{p}{n-1}} (\Rn, \Rnn)$ so, thanks to part \emph{\ref{item:wcontcof})}, $\cof D^s u_s \weakc \cof D u$ in $L^{\frac{p}{n-1}} (\Rn, \Rnn)$.
	By Lemma \ref{lem: weak covergence of the linear operator K}, $K^s_{\varphi} (\cof D^s u_s) \weakc \cof D u \, D\varphi$ in $L^r (\Rn, \Rn)$ for every $r \in (1, \frac{p}{n-1}]$.
	By Theorem \ref{Pr: weak convergence in s}, $u_s \to u$ in $L^t (\Rn)$ for every $t \in [1, \infty)$, so convergence \eqref{eq:ucofDu} holds since $p > n-1$.
	
	In either case, we have convergence \eqref{eq:ucofDu}, so by the analogue of \eqref{eq:recordarpartes1}--\eqref{eq:recordarpartes2} with $k=n$ we obtain
	\[
	\int \det D^s u_s (x) \, \varphi (x) \, dx \to \int \det D u(x) \, \varphi (x) \, dx .
	\]
	This shows that $\det D^s u_s \weakc \det D u$ in the sense of distributions, so $\theta = \det D u$.
\end{proof}

\section[$\Gamma$-convergence]{$\Gamma$-convergence}\label{se:Gamma}

$\Gamma$-convergence is the main conceptual tool for studying the variational convergence of families of functionals defined on metric spaces \cite{Braides}. In this section we show that the functional 
\[\mathcal{I}_s (u)=\int W(x,u(x),D^su(x))\,dx,\]
defined on $H^{s,p}(\Rn,\Rn)$, $\Gamma$-converges, as $s \nearrow 1$, to the functional
\[\mathcal{I} (u)=  \int W(x,u(x),Du(x))\,dx,\]
defined on $W^{1,p}(\Rn,\Rn)$ under the assumption of $W$ being polyconvex. We recall the concept of polyconvexity (see, e.g, \cite{Ball77,dacorogna}).
Let $\tau$ be the number of submatrices of an $n \times n$ matrix.
We fix a function $\vec \mu : \Rnn \to \R^{\tau}$ such that $\vec \mu (F)$ is the collection of all minors of an $F \in \Rnn$ in a given order.
A function $W_0 : \Rnn \to \R \cup \{\infty\}$ is polyconvex if there exists a convex $\Phi : \R^{\tau} \to \R \cup \{\infty\}$ such that $W_0 (F) = \Phi (\vec \mu (F))$ for all $F \in \Rnn$.
Polyconvexity of $W : \Rn \times \Rn \times \Rnn \to \R \cup \{ \infty \}$ means polyconvexity in the last variable.

It is convenient to consider both $\mathcal{I}_s$ and $\mathcal{I}$ defined on the same functional space independent of $s$, so we consider both functionals defined on $L^p(\Rn,\Rn)$.
The extension of $\mathcal{I}_s$ to $L^p (\Rn,\Rn) \setminus H^{s,p} (\Rn,\Rn)$ and of $\mathcal{I}$ to $L^p (\Rn,\Rn) \setminus W^{1,p} (\Rn,\Rn)$ is done by infinity.
Recalling the definition of $\Gamma$-convergence in this particular situation, we say that $\mathcal{I}_s$ $\Gamma$-converges to $\mathcal{I}$ as $s \nearrow 1$ in the strong topology of $L^p (\Rn, \Rn)$ if the following two conditions hold:
\begin{itemize}
	\item \textit{Liminf inequality:} For every family $\{u_s \}_{s\in (0,1)}$ in $L^p(\mathbb{R}^n,\Rn)$ such that $u_s \rightarrow u$ in $L^p(\mathbb{R}^n,\Rn)$ as $s \nearrow 1$, we have
	\begin{equation*}
	\mathcal{I} (u)\leq \liminf_{s \nearrow 1} \mathcal{I}_s (u_s).
	\end{equation*}
	\item \textit{Limsup inequality:} For each $u \in W^{1,p}(\mathbb{R}^n,\Rn)$, there exists a family $\{u_s \}_{s\in (0,1)}$ $\subset$ $L^p(\mathbb{R}^n,\Rn)$ such that $u_s \rightarrow u$ in $L^p(\mathbb{R}^n,\Rn)$ as $s \nearrow 1$ and
	\begin{equation*}
	\limsup_{s \nearrow 1} \mathcal{I}_s (u_s) \leq \mathcal{I} (u).
	\end{equation*}
\end{itemize}

Although not in the definition of $\Gamma$-convergence, it is customary to attach a compactness property to the conditions above, which, in this context, reads as follows:
\begin{itemize}
	\item \emph{Compactness:} For every $g \in W^{1,p} (\Rn, \Rn)$ and every family $\{ u_s \}_{s \in (0,1)}$ with $u_s = g$ in $\O^c$ for all $s \in (0,1)$ such that $\liminf_{s \nearrow 1} \mathcal{I}_s (u_s) < \infty$, there exist an increasing sequence $\{ s_j \}_{j \in \N}$ in $(0,1)$ with $\lim_{j \to \infty} s_j = 1$ and a $u \in L^p (\Rn, \Rn)$ such that $u_{s_j} \to u$ in $L^p (\Rn, \Rn)$ as $j \to \infty$.
\end{itemize}

The \emph{limsup inequality} will be a consequence of Theorem \ref{Prop: convergencia del gradiente fraccionario al clásico}, while the compactness property will follow of Theorem \ref{Pr: weak convergence in s}.
The \emph{liminf inequality}, on the other hand, is a novel semicontinuity result, which improves that of \cite{BeCuMC} done for a fixed $s$, and is singled out in the following proposition.
As we will see, the growth conditions for proving the \emph{liminf} and \emph{limsup} inequalitites are compatible only in the range $p > n$.

\begin{prop}\label{prop:lsc}
	Let $p \geq n-1$ satisfy $p>1$ and $0 < s < 1$. Let $\O$ be a bounded open subset of $\Rn$ and $g\in W^{1,p}(\Rn,\Rn)$. 
	Let $W : \Rn \times \Rn \times \Rnn \to \R \cup \{ \infty \}$ satisfy the following conditions:
	\begin{enumerate}[a)]
		
		\item $W$ is $\mathcal{L}^n \times \mathcal{B}^n \times \mathcal{B}^{n\times n}$-measurable, where $\mathcal{L}^n$ denotes the Lebesgue sigma-algebra in $\Rn$, whereas $\mathcal{B}^n$ and $\mathcal{B}^{n\times n}$ denote the Borel sigma-algebras in $\Rn$ and $\Rnn$, respectively.

\item $W (x, \cdot, \cdot)$ is lower semicontinuous for a.e.\ $x \in \Rn$.

		\item For a.e.\ $x \in \Rn$ and every $y \in \Rn$, the function $W (x, y, \cdot)$ is polyconvex.
		
		\item\label{item:Ecoerc}
		There exist a constant $c>0$, an $a \in L^1 (\Rn)$ and a Borel function $h : [0, \infty) \to [0, \infty)$ such that
		\[
		\lim_{t \to \infty} \frac{h (t)}{t} = \infty
		\]
		and
		\[
		\begin{cases}
		W (x, y, F) \geq a (x) + c \left| F \right|^p + c \left| \cof F \right|^q + h (\left| \det F \right|)  \quad \text{for some } q > \frac{p^*}{p^* -1},  & \text{if } p < n ,\\
		W (x, y, F) \geq a (x) + c \left| F \right|^p + h (\left| \det F \right|) , & \text{if } p = n , \\
		W (x, y, F) \geq a (x) + c \left| F \right|^p , & \text{if } p > n ,
		\end{cases}
		\]
		for a.e.\ $x \in \Rn$, all $y \in \Rn$ and all $F \in \Rnn$.
	\end{enumerate}
For each $s\in (0,1)$, let $u_s\in H^{s,p}_g(\O,\Rn)$ and $u\in W^{1,p} (\O,\Rn)$ satisfy $u_s\to u$ in $L^p(\O,\Rn)$ as $s \nearrow 1$.
Then 
		\begin{equation}\label{eq:liminf}
		\mathcal{I}(u) \leq \liminf_{s\nearrow 1} \mathcal{I}_s (u_s)
		\end{equation}
\end{prop}
\begin{proof}
We can assume that 
\begin{equation}\label{eq:Isusfinite}
\liminf_{s\nearrow 1} \mathcal{I}_s(u_s) < \infty ,
\end{equation}
hence by assumption \emph{\ref{item:Ecoerc})}, there exists an increasing sequence $\{ s_j \}_{j \in \N}$ in $(0,1)$ with $\lim_{j \to \infty} s_j = 1$ such that $\liminf_{s\nearrow 1} \mathcal{I}_s(u_s) = \lim_{j \to \infty} \mathcal{I}_{s_j} (u_{s_j})$ and the sequence $\{D^{s_j} u_{s_j} \}_{j \in \N}$ is bounded in $L^p(\Rn,\Rn)$, so by Theorem \ref{Pr: weak convergence in s}, for a subsequence (not relabelled),
	\begin{equation}\label{eq:convergence1poly}
	u_{s_j} \to u \quad \text{and} \quad D^{s_j} u_{s_j} \weakc Du  \qquad \text{in } L^p \text{ as } j \to \infty .
	\end{equation}
By Theorem \ref{th:wcontdet}, for any minor $\mu$ of order $k \leq n-2$, we have that
	\begin{equation}\label{eq:convergence2poly}
	\mu (D^{s_j} u_{s_j}) \weakc \mu (D u) \text{ in } L^{\frac{p}{k}} (\Rn) \text{ as } j \to \infty .
	\end{equation}
	
If $p <n$ then, by assumption \emph{\ref{item:Ecoerc})}, $\{ \cof D^s u_s \}_{0<s<1}$ is bounded in $L^q (\Rn, \Rnn)$, whereas if $p \geq n$ we set $q:= \frac{p}{n-1}$ and have that $\{ \cof D^s u_s \}_{0<s<1}$ is bounded in $L^q (\Rn, \Rnn)$.
	In either case we have that $q>1$, so for a subsequence $\{ \cof D^{s_j} u_{s_j}\}_{j \in \N}$ converges weakly in $L^q (\Rn, \Rnn)$ and, by Theorem \ref{th:wcontdet},
	\begin{equation}\label{eq:convergence3poly}
	\cof D^{s_j} u_{s_j} \weakc \cof D u \text{ in } L^q (\Rn, \Rnn) \text{ as } j \to \infty .
	\end{equation}
	
	If $p \leq n$ then, by assumption \emph{\ref{item:Ecoerc})} and de la Vall\'ee Poussin's criterion, $\{ \det D^s u_s \}_{0<s<1}$ is equiintegrable, whereas if $p > n$ we have that $\{ \det D^s u_s\}_{0<s<1}$ is bounded in $L^{\frac{p}{n}} (\Rn)$ and $\frac{p}{n} > 1$.
	In either case we have that, for a subsequence, $\{ \det D^{s_j} u_{s_j} \}_{j \in \N}$ converges weakly in $L^{\ell} (\Rn)$ with
	\[
	\begin{cases}
	\ell = 1 & \text{if } p \leq n , \\
	\ell = \frac{p}{n} & \text{if } p > n ,
	\end{cases}
	\]
	and, hence, by Theorem \ref{th:wcontdet},
	\begin{equation}\label{eq:convergence4poly}
	\det D^{s_j} u_{s_j} \weakc \det D u \text{ in } L^{\ell} (\Rn) \text{ as } j \to \infty .
	\end{equation}
	
	Convergences \eqref{eq:convergence1poly}--\eqref{eq:convergence4poly} imply, thanks to a standard lower semicontinuity result for polyconvex functionals (see, e.g., \cite[Th.\ 5.4]{BallCurrieOlver} or \cite[Th.\ 7.5]{FoLe07}), that for any $R>0$,
\begin{equation}\label{eq:BRlower}
	\int_{B(0, R)} W(x, u(x), D u(x)) \, dx \leq \liminf_{j\to \infty} \int_{B(0, R)} W(x, u_{s_j} (x), D^{s_j} u_{s_j} (x)) \, dx .
\end{equation}
	Therefore,
	\begin{align*}
	\int_{B(0, R)} \left[ W(x, u(x), D u(x)) - a(x) \right] dx \leq \liminf_{j\to \infty} \int \left[ W(x, u_{s_j}(x), D^{s_j} u_{s_j}(x)) - a(x) \right] dx .
	\end{align*}
	By monotone convergence,
	\[
	\int \left[ W(x, u(x), D u(x)) - a(x) \right] dx \leq \liminf_{j\to \infty} \int \left[ W(x, u_{s_j}(x), D^{s_j} u_{s_j}(x)) - a(x) \right] dx ,
	\]
	so 
	\[
	\mathcal{I}(u) \leq \liminf_{j\to \infty} \mathcal{I}_{s_j} (u_{s_j}),
	\]
as desired.
\end{proof}

We finally present the main result of this paper, which shows the $\Gamma$-convergence of polyconvex functionals defined on Bessel spaces, involving $s$-fractional gradients, to a classical local polyconvex functional defined on a Sobolev space.
Unfortunately, we crucially need the extra assumption $p>n$ in order to prove the \emph{limsup inequality}.
This is because the coercivity conditions of $W$ in Proposition \ref{prop:lsc} are compatible with the standard upper bound by $|F|^p$ (which makes the functional $\mathcal{I}$ continuous in $W^{1,p}$; see \cite{dacorogna}) only in the case $p>n$.

\begin{teo}\label{th:Gconvergence}
Let $p >n$ and $0 < s < 1$. Let $\O$ be a bounded open subset of $\Rn$ and $g\in W^{1,p}(\Rn,\Rn)$. 
	Let $W : \Rn \times \Rn \times \Rnn \to \R \cup \{ \infty \}$ satisfy the following conditions:
	\begin{enumerate}[a)]
		
\item $W$ is $\mathcal{L}^n \times \mathcal{B}^n \times \mathcal{B}^{n\times n}$-measurable.

\item $W (x, \cdot, \cdot)$ is lower semicontinuous for a.e.\ $x \in \Rn$.		

\item\label{item:poly} For a.e.\ $x \in \Rn$ and every $y \in \Rn$, the function $W (x, y, \cdot)$ is polyconvex.

\item Assume there exist $c> 0$ and $a \in L^1 (\Rn)$ such that
\[
 W (x, y, F) \geq a (x) + c \left| F \right|^p , \qquad \text{a.e. } x \in \Rn, \text{ all } y \in \Rn , \text{ all } F \in \Rnn ,
\]
and for every $R>0$ there exist $a_R \in L^1 (\Rn)$ and $c_R >0$ such that for a.e.\ $x \in \Rn$, all $y \in \Rn$ with $\left| y \right| \leq R$ and all $F \in \Rnn$,
\[
 W (x, y, F) \leq a_R (x) + c_R \left| F \right|^p .
\]

\end{enumerate}
The following statements hold:
	\begin{enumerate}[i)]
		\item\label{item:compactness} For each $s\in (0,1)$, let $u_s\in H^{s,p}_g (\O,\Rn)$ satisfy
		\begin{equation}\label{eq:supbound} 
		\liminf_{s \nearrow 1} \mathcal{I}_s(u_s) < \infty.
		\end{equation}
		Then there exist $u\in W^{1,p}_g (\O, \Rn)$ and an increasing sequence $\{ s_j \}_{j \in \N}$ in $(0,1)$ with $\lim_{j \to \infty} s_j = 1$ such that $u_{s_j} \to u$ in $L^p(\R^n,\Rn)$ as $j \to \infty$.

		\item \label{item:liminf} For each $s\in (0,1)$, let $u_s\in H^{s,p}_g(\O, \Rn)$ and $u\in W^{1,p} (\Rn, \Rn)$ satisfy $u_s\to u$ in $L^p(\Rn, \Rn)$. Then 
		\[
		\mathcal{I}(u) \leq \liminf_{s\nearrow 1} \mathcal{I}_s (u_s) . 
		\]
		\item \label{item:limsup} For each $u\in W^{1,p}_g(\O,\Rn)$ and $s\in (0,1)$, there exists $u_s\in H^{s,p}_g(\O,\Rn)$ such that $u_s\to u$ in $L^p(\Rn,\Rn)$ and 
		\begin{equation}\label{eq:limsup}
		\limsup_{s\nearrow 1} \mathcal{I}_s (u_s) \leq \mathcal{I} (u).
		\end{equation}
	\end{enumerate}
\end{teo}
\begin{proof}
	For proving \emph{\ref{item:compactness})}, just notice that by assumption \emph{\ref{item:Ecoerc})}, \eqref{eq:supbound} implies that there is an increasing sequence $\{ s_j \}_{j \in \N}$ in $(0,1)$ with $\lim_{j \to \infty} s_j = 1$ such that $\{ D^{s_j} u_{s_j} \}_{j \in \N}$ is bounded in $L^p(\Rn,\Rn)$.
Therefore, by Theorem \ref{Pr: weak convergence in s}, there exists $u\in W^{1,p}_g (\O,\Rn)$ such that, for a subsequence $u_{s_j} \to u$ in $L^p(\Rn,\Rn)$ as $j \to \infty$.
	
Part \emph{\ref{item:liminf})} is a particular case of Proposition \ref{prop:lsc}. 
	
Finally we show \emph{\ref{item:limsup})}, so we let $u \in W^{1,p}_g (\O, \Rn)$.
By Theorem \ref{Prop: convergencia del gradiente fraccionario al clásico}, $D^s u \to D u$ in $L^p (\Rn , \Rnn)$ as $s \nearrow 1$.
Assumption \emph{\ref{item:poly})} implies in particular the continuity of $W (x, y, \cdot)$ for a.e.\ $x \in \Rn$ and all $y \in \Rn$ (see, e.g., \cite{dacorogna}).
The Sobolev embedding shows that $u$ is bounded.
By the growth conditions and dominated convergence, 
\begin{equation}\label{eq:limW}
 \lim_{s \nearrow 1} \int W (x, u, D^s u) = \int W (x, u, D u) ,
\end{equation}
which proves \eqref{eq:limsup}.
\end{proof}

Although the bulk of this article has been focused on the assumption of polyconvexity, with the stronger assumption of convexity we can achieve the analogue result of Theorem \ref{th:Gconvergence} for the full range of exponents $p \in (1, \infty)$.
Since the proof is analogous (and in some steps, simpler) than that of Theorem \ref{th:Gconvergence}, it will only be sketched.

\begin{teo}\label{th:GconvergenceC}
Let $1 < p < \infty$ and $0 < s < 1$.
Let $\O$ be a bounded open subset of $\Rn$ and $g\in W^{1,p}(\Rn,\Rn)$. 
Let $W : \Rn \times \Rn \times \Rnn \to \R \cup \{ \infty \}$ satisfy the following conditions:
\begin{enumerate}[a)]
\item $W$ is $\mathcal{L}^n \times \mathcal{B}^n \times \mathcal{B}^{n\times n}$-measurable.

\item $W (x, \cdot, \cdot)$ is lower semicontinuous for a.e.\ $x \in \Rn$.		

\item For a.e.\ $x \in \Rn$ and every $y \in \Rn$, the function $W (x, y, \cdot)$ is convex.

\item[d1)] If $p<n$ assume there exist $c\geq 1$ and $a \in L^1 (\Rn)$ such that for a.e.\ $x \in \Rn$, all $y \in \Rn$ and all $F \in \Rnn$,
\[
 - a (x) + \frac{1}{c} \left| F \right|^p \leq W (x, y, F) \leq a (x) + c \left( \left| y \right|^p + \left| y \right|^{p^*} + \left| F \right|^p \right) .
\]

\item[d2)] If $p=n$ assume there exist $r \in [p, \infty)$, $c\geq 1$ and $a \in L^1 (\Rn)$ such that for a.e.\ $x \in \Rn$, all $y \in \Rn$ and all $F \in \Rnn$,
\[
 - a (x) + \frac{1}{c} \left| F \right|^p \leq W (x, y, F) \leq a (x) + c \left( \left| y \right|^p + \left| y \right|^r + \left| F \right|^p \right) .
\]

\item[d3)] If $p>n$ assume there exist $c> 0$ and $a \in L^1 (\Rn)$ such that
\[
 W (x, y, F) \geq a (x) + c \left| F \right|^p , \qquad \text{a.e. } x \in \Rn, \text{ all } y \in \Rn , \text{ all } F \in \Rnn ,
\]
and for every $R>0$ there exist $a_R \in L^1 (\Rn)$ and $c_R >0$ such that for a.e.\ $x \in \Rn$, all $y \in \Rn$ with $\left| y \right| \leq R$ and all $F \in \Rnn$,
\[
 W (x, y, F) \leq a_R (x) + c_R \left| F \right|^p .
\]

\end{enumerate}
Then, statements \ref{item:compactness})--\ref{item:limsup}) of Theorem \ref{th:Gconvergence} hold.
\end{teo}
\begin{proof}
The proof of \emph{\ref{item:compactness})} is the same as that of Theorem \ref{th:Gconvergence}.
	
For the proof of \emph{\ref{item:liminf})} we initially follow that of Proposition \ref{prop:lsc}.
We can assume inequality \eqref{eq:Isusfinite}, so there exists an increasing sequence $\{ s_j \}_{j \in \N}$ in $(0,1)$ with $\lim_{j \to \infty} s_j = 1$ such that $\liminf_{s\nearrow 1} \mathcal{I}_s(u_s) = \lim_{j \to \infty} \mathcal{I}_{s_j} (u_{s_j})$ and the sequence $\{D^{s_j} u_{s_j} \}_{j \in \N}$ is bounded in $L^p(\Rn,\Rn)$.
By Theorem \ref{Pr: weak convergence in s}, for a subsequence, convergences \eqref{eq:convergence1poly} hold.
By a standard lower semicontinuity result for convex functionals (see, e.g.,\cite[Th.\ 7.5]{FoLe07}), for any $R>0$, inequality \eqref{eq:BRlower} holds, and we conclude \eqref{eq:liminf} as in Proposition \ref{prop:lsc}.

To show \emph{\ref{item:limsup})} we apply Theorem \ref{Prop: convergencia del gradiente fraccionario al clásico} and obtain $D^s u \to D u$ in $L^p (\Rn , \Rnn)$ as $s \nearrow 1$.
Assumption \emph{\ref{item:poly})} implies in particular the continuity of $W (x, y, \cdot)$ for a.e.\ $x \in \Rn$ and all $y \in \Rn$.
The Sobolev embedding in the three cases ($p<n$, $p=n$ and $p>n$) shows that the growth conditions allow us to apply dominated convergence and conclude inequality \eqref{eq:limW}, as desired.
\end{proof}

 \section*{Acknowledgements}

This work has been supported by the Agencia Estatal de Investigaci\'on of the Spanish Ministry Research and Innovation, through projects MTM2017- 83740-P (J.C.B. and J.C.), and MTM2017-85934-C3-2-P (C.M.-C.).

%
%\bibliography{bibliography}{}
%\bibliographystyle{plain} 
\end{document}